\documentclass[11pt,reqno]{amsart}
\usepackage{xifthen}
\usepackage{subfig}
\usepackage{pkgfile}
\usepackage{epsfig,latexsym,graphicx}
\usepackage{algorithm2e}
\newcommand{\bbV}{{\mathcal V}}
\newcommand{\bbE}{{\mathcal E}}
\newtheorem{obs}{Observation}[section]

\newtheorem{definition}{Definition}

\begin{document}

\title[Minimum-Weight Edge Discriminator in Hypergraphs]{Minimum-Weight Edge Discriminator in Hypergraphs}
\author{Bhaswar B. Bhattacharya}
\address{Department of Statistics, 
Stanford University, California, USA}
\email{bhaswar@stanford.edu}
\author{Sayantan Das}
\address{Department of Biostatistics, School of Public Health, 
University of Michigan, Ann Arbor, USA}
\email{sayantan@umich.edu}
\author{Shirshendu Ganguly}
\address{Department of Mathematics, 
University of Washington, Seattle, USA}
\email{sganguly@math.washington.edu}

\begin{abstract}
In this paper we introduce the concept of minimum-weight edge-discriminators in hypergraphs, and study its various properties. For a hypergraph $\mathcal H=(\mathcal V, \mathcal E)$, a function $\lambda: \mathcal V\rightarrow \mathbb Z^{+}\cup\{0\}$ is said to be an {\it edge-discriminator} on $\mathcal H$ if $\sum_{v\in E_i}{\lambda(v)}>0$, for all hyperedges $E_i\in \mathcal E$, and $\sum_{v\in E_i}{\lambda(v)}\ne \sum_{v\in E_j}{\lambda(v)}$, for every two distinct hyperedges $E_i, E_j \in \mathcal E$. An {\it optimal edge-discriminator} on $\mathcal H$, to be denoted by $\lambda_\mathcal H$, is an edge-discriminator on $\mathcal H$ satisfying $\sum_{v\in \mathcal V}\lambda_\mathcal H (v)=\min_\lambda\sum_{v\in \mathcal V}{\lambda(v)}$, where the minimum is taken over all edge-discriminators on $\mathcal H$.  We prove that any hypergraph $\mathcal H=(\mathcal V, \mathcal E)$,  with $|\mathcal E|=n$, satisfies $\sum_{v\in \mathcal V} \lambda_\mathcal H(v)\leq n(n+1)/2$, and equality holds if and only if the elements of $\mathcal E$ are mutually disjoint. For $r$-uniform hypergraphs $\mathcal H=(\mathcal V, \mathcal E)$, it follows from results on Sidon sequences that $\sum_{v\in \mathcal V}\lambda_{\mathcal H}(v)\leq |\mathcal V|^{r+1}+o(|\mathcal V|^{r+1})$, and the bound is attained up to a constant factor by the complete $r$-uniform hypergraph.  Next, we construct optimal edge-discriminators for some special hypergraphs, which include paths, cycles, and complete $r$-partite hypergraphs. Finally, we show that no optimal edge-discriminator on any hypergraph $\mathcal H=(\mathcal V, \mathcal E)$, with $|\mathcal E|=n~(\geq 3)$, satisfies $\sum_{v\in \mathcal V} \lambda_\mathcal H (v)=n(n+1)/2-1$. This shows that not all integer values between $n$ and $n(n+1)/2$ can be the weight of an optimal edge-discriminator of a hypergraph, which, in turn, raises many other interesting combinatorial questions.
\end{abstract}


\maketitle

\section{Introduction}

A {\it hypergraph} is a pair $\mathcal H=(\mathcal V, \mathcal E)$ where $\mathcal V$ is a finite set and $\mathcal E$ is a collection of subsets of $\mathcal V$. The elements of $\mathcal V$ are called {\it vertices} and the elements of $\mathcal E$ are called {\it hyperedges}. A vertex labeling of a hypergraph is a function from the vertex set $\mathcal V$ to the set of non-negative integers. In this paper, we introduce the notion of edge-discriminating vertex labelings in hypergraphs. A labeling $\lambda: \mathcal V\rightarrow \mathbb Z^{+}\cup\{0\}$ is said to be an {\it edge-discriminator} on $\mathcal H$ if $\sum_{v\in E_i}{\lambda(v)}>0$, for all hyperedges $E_i\in \mathcal E$, and $\sum_{v\in E_i}{\lambda(v)}\ne \sum_{v\in E_j}{\lambda(v)}$, for every two distinct hyperedges $E_i, E_j \in \mathcal E$. For any edge-discriminator $\lambda$ on $\mathcal H$, the value of the sum $\sum_{v\in \mathcal V}\lambda(v)$ will be called the {\it weight} of the edge-discriminator and shall be denoted by $\omega_\lambda(\mathcal V)$. An edge-discriminator $\lambda_\mathcal H$ on $\mathcal H$ is said to be an {\it optimal edge-discriminator} if it has the least weight, that is, if $\omega_{\lambda_\mathcal H}(\mathcal V)=\min_\lambda\omega_\lambda(\mathcal V)$, where the minimum is taken over all edge-discriminators on $\mathcal H$. Henceforth, the weight of the optimal edge-discriminator on $\mathcal H$, that is, $\omega_{\lambda_\mathcal H}(\mathcal V)$, will be denoted by $\omega_0(\mathcal H)$.

In this paper we prove several properties of hypergraph edge-discriminators and explicitly compute optimal edge-discriminators in some special types of hypergraphs.

\subsection{Related Works}

Hypergraph vertex labelings such that the sum of the labels of the vertices along the edges are mutually distinct, has been studied in the literature in many different contexts. One of them is the notion of anti-magic labeling on graphs. In general, graph labeling is an assignment of integers to the vertices or edges, or both, of a graph which satisfy certain conditions. Refer to the survey of Gallian \cite{graphlabelingsurvey} for a comprehensive view into the colossal literature in graph labeling. For a graph $G=(V, E)$ an {\it edge-antimagic vertex labeling} $l:V\rightarrow \mathbb Z^+$ is a injective function such that the quantities $l(x)+l(y)$ are mutually distinct, whenever $(x, y)$ is an edge in $G$. Edge-antimagic vertex labeling was studied by Wood \cite{wood}. Later Bollob\'as and Pikhurko \cite{bollobaspikhurko} defined the {\it sum magic number} of a graph $G=(V(G), E(G))$, denoted by $\mathcal S(G)$, as the smallest value of the largest vertex label in an edge-antimagic vertex labeling. They proved that $\mathcal S(K_n)=(1+o(1))n^2$ and $\mathcal S(n, m)< (1-c)n^2$ whenever $m \leq c n^2$, where $\mathcal S(n, m)=\max\{\mathcal S(G): |V(G)|=n, |E(G)|=m\}$. 
Note that, unlike in the case of edge-discriminators, edge anti-magic vertex labeling on a graph is an injective function. Moreover, the sum magic number minimizes the maximum label, and not the the sum of the labels as required in the optimal edge-discriminator.

Another relevant line of research involves the notion of irregular networks. A {\it network} is a simple graph where each edge is assigned a positive integer weight. The {\it degree} of a vertex in a network is the sum of the weights of its incident edges. A network is {\it irregular} if all the vertices have distinct degrees. The strength of a network is the maximum weight assigned to any edge. The irregularity strength of a graph $G$ is the minimum strength among all irregular network on $G$, and is denoted by $s(G)$. The notion of irregularity strength was first introduced by Chartrand et al. \cite{chartrand}, where it was shown that for any given graph $G$, $$s(G)\geq \lambda(G) = \max_{i\leq j}\frac{(n_i + n_{i+1} + \ldots + n_{j}) + i - 1}{j},$$
where $n_i$ denotes the number of vertices of degree $i$. If $G$ contains a $K_2$ or multiple isolated vertices, the irregularity strength $s(G)=\infty$. Nierhoff \cite{nierhoff} proved a tight bound $s(G)\leq m-1$, for graphs $G$ with $|V(G)|=m~(>3)$ and $s(G) < \infty$. Faudree and Lehel \cite{faudreelehel} proved bounds on the irregularity strength of $d$-regular graphs. Upper bounds on the irregularity strength for general graphs in terms of the minimum degree was first given by Frieze et al. \cite{frieze}, and later by Przyby\l o \cite{przybylo,przybylosiam} and Kalkowski et al. \cite{pfender}. However, computing the irregularity strength of a graph exactly is difficult in general. It is known only for some very special graphs and in almost all of these cases, it is found to be within an additive constant of $\lambda(G)$. It was conjectured by Lehel \cite{lehel}, that $s(G)$ is within an additive constant of $\lambda(G)$ for connected graphs. This has been verified for some special families of graphs like, complete graphs \cite{chartrand}, cycles, most complete bipartite graphs, Turan graphs \cite{irregular}, wheels, hypercubes, and grids \cite{ebert}. The problem of studying the irregularity strength was extended for hypergraphs by Gy\'arf\'as et al. \cite{jcmcc}.

Edge-discriminators on hypergraphs and irregular networks are dual concepts. For a hypergraph $\mathcal H=(\mathcal V, \mathcal E)$, the {\it dual hypergraph} is defined as $\mathcal D(\mathcal H)=(\mathcal W, \mathcal F)$, where $\mathcal W=\mathcal E$ and $\mathcal F=\{\mathcal E(v)|v\in V\}$, where $\mathcal E(v)$ is the set of all edges in $\mathcal E$ which are incident on $v\in \mathcal V$. Note that if $\kappa: \mathcal E\rightarrow \mathbb Z^+$ is an irregular edge assignment for $\mathcal H$, then $\kappa$ transforms to an edge-discriminator $\lambda: \mathcal W\rightarrow \mathbb Z^+$ on $\mathcal D(\mathcal H)$ as follows: For every vertex $v\in \mathcal W$, let $e_v\in \mathcal E$ be the corresponding hyperedge in $\mathcal H$ and define $\lambda(v)=\kappa(e_v)$. However, the most important difference between irregularity strength of a hyper graph and the optimal edge-discriminator in the dual hypergraph is the optimization criterion. In the case of irregularity strength the maximum label is minimized, whereas we minimize the sum of the labels in the optimal edge-discriminator. Another difference is that in an irregular network, the value assigned to an edge is always positive, which means that the edge-discriminator in the corresponding hypergraph assigns a positive weight to every vertex, which is not required in the definition of an edge-discriminator.

Another related line of work exists in the context of the  power set hypergraph. The {\it power set hypergraph} on a set $\mathcal V$, with $|\mathcal V|=m$, is the hypergraph $(\mathcal V, 2^\mathcal V)$, where $2^\mathcal V$ denotes the set of all non-empty subsets of $\mathcal V$. Note that any edge-discriminator on the power set hypergraph is a set of positive integers such that all its non-empty subsets have distinct sums. A set of positive numbers satisfying this property is called {\it sum-distinct}. A sum-distinct set of $m$ elements with the minimum total sum is the optimal edge-discriminator on the power set hypergraph. This can be easily computed as we shall show in Section \ref{sec:power_set}. In 1931 Erd\H{o}s asked for estimates of smallest possible value of the largest element in a sum-distinct set of $m$ elements, which we denote by $w(m)$. Erd\H{o}s offered 500 dollars for verifying whether $w(m)=\Omega(2^m)$, and Guy \cite{guy} made the stronger conjecture that $w(m) > 2^{m-3}$. In 1955 Erd\H{o}s and Moser proved that $w(m) \geq 2^m/(4\sqrt m)$ \cite{erdos_sum_distinct}. The constant was later improved by Elkies \cite{elkies}, which was further improved by Aliev \cite{aliev}. A set consisting of the first $m$ powers of 2 has distinct subset sums, and has maximal element $2^{m-1} $, which implies that $w(m) \leq 2^{m-1} $. Conway and Guy \cite{conwayguy} found a construction of sum-distinct sets which gave an interesting upper bound on $w(m)$. This was later improved by Lunnon \cite{lunnon} and then by Bohman \cite{bohman}, who showed that $w(m) < 0.22002 \cdot 2^m$ for sufficiently large $m$.

\subsection{Our Results}

In this paper we study edge-discriminators on hypergraphs such that the sum of the labels of the vertices is minimized. We begin by proving a general upper bound on the weight of an edge-discriminator, which holds for any hypergraph. The bound is relatively simple to obtain, but it is tight.

\begin{thm}
For any hypergraph $\mathcal H=(\mathcal V, \mathcal E)$, with $|\mathcal E|=n$, $\omega_0(\mathcal H)\leq n(n+1)/2$, and equality holds if and only if the elements of $\mathcal E$ are mutually disjoint. 
\label{th:hypergraph_main}
\end{thm}

Next, we show that the edge-discrimination problem for $r$-uniform hypergraphs is related to Sidon sequences from additive number theory. A {\it Sidon sequence} is a sequence of natural numbers $A = \{a_1, a_2, \ldots \}$ such that all the pairwise sums $a_i + a_j$ $(i \leq j)$ are different \cite{erdosturansidon}. $B_h$-{\it sets} are generalizations of Sidon-sequences in which all $h$-element sums are mutually distinct \cite{sidonsurvey}. Using the connection between edge-discriminators and $B_h$ sets, we obtain another bound on the weight of the optimal edge-discriminator for $r$-uniform hypergraphs in terms of the number of vertices of the hypergraph.

\begin{ppn}For any $r$-uniform hypergraph $\mathcal H=(\bbV, \bbE)$, with $|\mathcal V|=m$, $\omega_0(\mathcal H)\leq m^{r+1}+o(m^{r+1})$, and the bound is attained up to a constant factor by the optimal edge-discriminator of the complete $r$-uniform hypergraph on $\mathcal V$. 
\label{th:r+1}
\end{ppn}

Obtaining nontrivial lower bounds on the weight of the optimal edge-discriminator for general hypergraphs is a major challenge. It is easy to show that $\sum_{v\in \mathcal V} \lambda_\mathcal H (v)\geq \max\{n, \delta(\delta+1)/2\}$, where $|\mathcal E|=n$ and $\delta$ is the size of the maximum matching in $\mathcal H$. Moreover, there is a hypergraph which attains this lower bound. However, like the irregularity strength, finding the optimal edge-discriminators is generally difficult even for very special hypergraphs. Nevertheless, we were able to obtain optimal edge-discriminators for some specific hypergraphs, which include paths, cycles and the complete bipartite graph. Constructing these optimal edge-discriminators are in itself interesting combinatorial problems and the results we obtain are summarized below.


\begin{thm}$\omega_0(P_m)=\lceil m(m-1)/4\rceil$, where $P_m$ is the path with $m$ vertices. 
\label{th:path}
\end{thm}

Using the construction of the optimal edge-discriminator for paths we construct the optimal edge-discriminators for cycles.

\begin{thm}$\omega_0(C_m)=\lceil m(m+1)/4 \rceil$, where $C_m$ is the cycle with $m$ vertices. 
\label{th:cycle}
\end{thm}

Next, we consider the optimal edge-discriminator for the {\it complete $r$-partite hypergraph}, which is a generalization of the complete bipartite graph. Let $A_1, A_2, \ldots, A_r$ be disjoint sets with $m_1, m_2, \ldots, m_r$ elements, respectively, where $m_1\geq m_2\geq\ldots\geq m_r$ are positive integers. If $\pmb a=(m_1, m_2, \ldots, m_r)$, the hypergraph $\mathcal H_r(\pmb a):=(\mathcal V_r, \bbE_r)$, with $\mathcal V_r= \bigcup_{i=1}^r A_i$ and 
$\bbE_r=A_1\times A_2\times\ldots \times A_r$ is called the {\it complete $r$-partite hypergraph}. We show that

\begin{thm}$\omega_0(\mathcal H_r(\pmb a))=m_r+\frac{1}{2}\cdot\sum_{q=1}^r(m_q-1)\prod_{s=1}^q m_s$. For the complete bipartite graph $K_{p, q}~ (p \geq q)$, $\omega_0(K_{p, q})=\frac{2q+p(p-1+q(q-1))}{2}$.

\label{th:r-partite}
\end{thm}


In Theorem \ref{th:hypergraph_main} we show that the weight of an optimal edge-discriminator for a hypergraph with $n$ hyperedges is at most $n(n+1)/2$. Moreover, the weight of any edge-discriminator is at least $n$. This motivates us to ask the following question: Given any integer $w\in [n, n(n+1)/2]$, whether there exists a hypergraph $\mathcal H$ with $n$ hyperedges such that $w$ is the weight of the optimal edge-discriminator on $\mathcal H$. We show that the answer is no, by proving the following theorem:

\begin{thm}There exists no hypergraph $\mathcal H=(\mathcal V, \mathcal E)$, with $|\mathcal E|=n~(\geq 3)$, such that the weight of the optimal edge-discriminator on $\mathcal H$ is $n(n+1)/2-1$. 
\label{th:no_function}
\end{thm}

This shows that the problem of attainability of weights is an interesting combinatorial problem which might have surprising consequences. We discuss the attainability problem in more details later on. 

The paper is organized as follows: The proof of Theoerem \ref{th:hypergraph_main} where we give a general upper bound on the weight of an optimal edge-discriminator is given in Section \ref{sec:upperbound}. In Section \ref{sec:sidon} we outline the connection between Sidon sequences and edge-discriminators in uniform hypergraphs. A short discussion on lower bounds is given in Section \ref{sec:lowerbound}. The computations of the optimal edge-discriminators for special hypergraphs is in Section \ref{sec:computation}. The problem of non-attainable weights and the proof of Theorem \ref{th:no_function} is in Section \ref{sec:nonattainable}. In Section \ref{sec:application} we discuss edge-discriminators in geometric hypergraphs and its potential application to digital image indexing. Finally, in Section \ref{sec:conclusions} we summarize our work and give directions for future research.




\section{Constructing Edge-Discriminators in General Hypergraphs}
\label{sec:upperbound}

In this section, we shall give an algorithm for constructing an edge-discriminator for a general hypergraph, using which we shall prove Theorem \ref{th:hypergraph_main}. We begin by introducing some notations. Consider a hypergraph $\mathcal H=(\mathcal V, \mathcal E)$, with $|\mathcal V|=m$ and $\mathcal E=\{ E_1,  E_2, \ldots,  E_n\}$, where $|\bbE|=n$ and $E_i \subset \mathcal V$ for $i\in [n]:=\{1, 2, \ldots, n\}$. An {\it ordering} on $\mathcal V$ is a bijective function $\nu: [m] \rightarrow \bbV$. We shall write $\nu_i:=\nu(i)$, for $i \in [m]$. Thus, with respect to the ordering $\nu$, the vertices in $\bbV$ will be indexed as $\{\nu_1, \nu_2, \ldots, \nu_m\}$. For two vertices $\nu_i, \nu_j \in \mathcal V$, we say $\nu_i$ {\it is less than} $\nu_j$ with respect to $\nu$, if $i < j$. The {\it maximal vertex} of $\mathcal W\subset \mathcal V$ is the vertex $\nu_k\in \mathcal W$ such that for all vertices $\nu_i \in \mathcal W\backslash\{\nu_k\}$, we have $i<k$. The maximal vertex of $\mathcal W$ will be denoted by $\nu(\mathcal W)$. For two hyperedges $E_i, E_j\in \mathcal E$ ($i\ne j$) the vertex $\nu(E_i\Delta E_j):=\nu(E_i, E_j)$ will be called the {\it differentiating vertex} of the edges $E_i$ and $E_j$, where for any two sets $A$ and $B$, $A\Delta B=(A\backslash B)\cup (B\backslash A)$. The edge which contains the differentiating vertex $\nu(E_i, E_j)$ will be denoted by $\overline{E}_{ij}$, and the edge which does not contain the differentiating vertex will be denoted by $\underline{E}_{ij}$. 

For any function $\lambda:\mathcal V\rightarrow \mathbb{Z}^{+}\cup \{0\}$, the {\it weight} of any subset $\mathcal W$ of $\mathcal V$ is defined as $\omega_\lambda(\mathcal W)=\sum_{v\in \mathcal W}{\lambda(v)}$. Thus, a function $\lambda$ is edge-discriminating if the weights of all the edges in $\bbE$ are distinct. 


\subsection{Proof of Theorem \ref{th:hypergraph_main}}


Given the hypergraph $\mathcal H=(\bbV, \bbE)$, with $|\bbV|=m$ and $|\bbE|=n$, for which we need to construct an edge-discriminator, consider the hypergraph $\mathcal H_0=(\bbV, \mathcal F)$, where $\mathcal F=\bbE\cup\{\emptyset\}$. Fix an ordering $\nu$ on $\mathcal V$. Let $\lambda: \mathcal V\rightarrow \mathbb Z^+\cup\{0\}$ be a function initialized as $\lambda(v)=0$, for all $v\in \mathcal V$. We iteratively update the value of the function at the vertices in the ordering induced by $\nu$. Abusing notation we will denote the function by $\lambda$ throughout the iterative procedure. Once $\lambda(\nu_1), \lambda(\nu_2), \ldots, \lambda(\nu_{k-1})$ are updated, update $\lambda(\nu_{k})$ by adding the least non-negative integer not in the set $$\mathcal A(\nu_k)=\{\omega_\lambda(\underline{E}_{ij})-\omega_\lambda(\overline{E}_{ij}):\nu( E_i, E_j)=\nu_k, E_i, E_j \in \mathcal F\}.$$ This implies that $\lambda(\nu_{k})$ is at most $|\mathcal A(\nu_k)|$, since initially $\lambda(\nu_{k})$ was $0$ .
Note that for any two hyperedges $E_i, E_j \in \mathcal F$, such that $\nu( E_i, E_j)=\nu_k$,  
all the vertices which are greater than $\nu_k$ are common to both or belongs to neither of the edges $E_i$ and $E_j$. Therefore, according to the above construction $\omega_\lambda(\underline{E}_{ij})-\omega_\lambda(\overline{E}_{ij})$ cannot change once $\lambda(\nu_k)$ is assigned. Hence, by the choice of $\lambda(\nu_k)$, $\omega_\lambda(E_i)\ne \omega_\lambda(E_j)$, from the $k$-th step onwards and therefore eventually. In particular on taking one of the edges in the pair to be $\emptyset$ this  implies that eventually $\omega_\lambda(E_i)>0$ for all $E_i\in \mathcal E$, as $\omega_\lambda(\emptyset)=0$. Therefore at the end of the updating procedure the function $\lambda$ is an edge-discriminator on $\mathcal H=(\mathcal V, \mathcal E)$. 

Next, observe that $|\mathcal A(\nu_k)|$ is at most the number of pairs of edges in $\mathcal F$ which has $\nu_k$ as a differentiating vertex. As $\nu(E_i, \emptyset)=\nu(E_i)$, it immediately follows that 
$|\mathcal A(\nu_k)|\leq \chi(\nu_k)+\pi(\nu_k)$, where $\chi(\nu_k)$ denotes the number of pairs of edges in $\bbE$ for which $\nu_k$ is a differentiating vertex, and $\pi(\nu_k)$ is the number of edges in $\bbE$ for which $\nu_k$ is the maximal vertex. This implies that 
\begin{eqnarray}
\sum_{v\in \mathcal V}\lambda(v)=\sum_{k=1}^{|\mathcal \bbV|} \lambda(\nu_{k}) &\leq&\sum_{k=1}^{m}\chi(\nu_{k}) + \sum_{k=1}^{m} \pi(\nu_{k})\nonumber \\
&=& \frac{n(n-1)}{2}+n=\frac{n(n+1)}{2},
\label{eqn:n(n+1)/2_main_II}
\end{eqnarray}
which completes the proof of the first part of Theorem \ref{th:hypergraph_main}.

\subsubsection{General Algorithm for Constructing Edge-Discriminators}

Before we complete the proof of Theorem \ref{th:hypergraph_main} we generalize our procedure for constructing an edge-discriminator. This digression is needed to complete the proof of
Theorem \ref{th:hypergraph_main}. Moreover, it also gives better insight in to the structure of edge-discriminating functions, and ultimately leads to a slightly improved upper bound on the weight of an optimal edge-discriminator.

The general algorithm is very similar to the algorithm in the proof of Theorem \ref{th:hypergraph_main}, but instead of starting with the zero function, we start with any arbitrary integer valued function $\kappa$, and execute the same procedure to get an edge-discriminator. Typically, depending on the structure of the hypergraph one can choose the initial function to obtain sharper bounds. More formally, the {\it Construction Algorithm} takes a hypergraph $\mathcal H=(\mathcal V, \mathcal E)$, an ordering $\nu$ on $\mathcal V$, and any function $\kappa: \mathcal V\rightarrow \mathbb Z^+\cup\{0\}$, and returns an edge-discriminator $\lambda_{\kappa, \nu}: \mathcal V\rightarrow \mathbb Z^+\cup\{0\}$. 

\begin{algorithm}[h]
\caption{{\it Construction Algorithm}: Edge-Discriminator Construction Algorithm.}
\hrule
\KwIn{A hypergraph $\mathcal H=(\mathcal V, \mathcal E)$, an ordering $\nu$ on $\mathcal V$, and any function $\kappa: \mathcal V\rightarrow \mathbb Z^+\cup\{0\}$.}
\KwOut{An edge-discriminator $\lambda_{\kappa, \nu}: \mathcal V\rightarrow \mathbb Z^+\cup\{0\}$.}
Initialize $\lambda_{\kappa, \nu}(v)=\kappa(v)$, for all $v\in \mathcal V$.\\
\While{$1< k\leq m$}{
Denote $\mathcal B(\nu_k)=\{|\omega_{\lambda_{\kappa, \nu}}(\underline{E}_{ij})-\omega_{\lambda_{\kappa, \nu}}(\overline{E}_{ij})|:\nu( E_i, E_j)=\nu_k, E_i, E_j \in \mathcal E\}$, $\mathcal C(\nu_k)=\{\omega_{\lambda_{\kappa, \nu}}(E_{i}):\nu( E_i, \emptyset)=\nu_k, E_i \in \mathcal E\}$, and $\mathcal A(\nu_k)=\mathcal B(\nu_k)\cup\mathcal C(\nu_k)$\;

\eIf{$|\mathcal A(\nu_k)|\ne 0$}{
\If{$\kappa(\nu_k)=0$}
{Define $\lambda_{\kappa, \nu}(\nu_{k})=\inf\{\mathbb Z^+\cup\{0\}\backslash\mathcal A(\nu_k)\}$,
}
\If{$\kappa(\nu_k)> 0$}
{
Define $\lambda_{\kappa, \nu}(\nu_{k})=\kappa(\nu_k)+\inf\{\mathbb Z^+\cup\{0\}\backslash\mathcal B(\nu_k)\}$.
}
}
{Define $\lambda_{\kappa, \nu}(\nu_{k})=\kappa(\nu_{k})$,
}
$k\leftarrow k+1$.
}
Return the function $\lambda_{\kappa, \nu}: \mathcal V\rightarrow \mathbb Z^+\cup\{0\}$.
\hrule
\label{algo:aldo_ed_discrimination}
\end{algorithm}


\begin{lem}
Given any hypergraph $\mathcal H=(\mathcal V, \mathcal E)$, an ordering $\nu$ on $\mathcal V$, and any function $\kappa: \mathcal V\rightarrow \mathbb Z^+\cup\{0\}$, the {\em Construction Algorithm} produces
an edge-discriminator $\lambda_{\kappa, \nu}: \mathcal V\rightarrow \mathbb Z^+\cup\{0\}$ on $\mathcal H$ such that $\omega_{\lambda_{\kappa, \nu}}(\mathcal V)\leq \frac{n(n+1)}{2}-\sum_{k=1}^m\pi(\nu_k)\pmb1\{\kappa(\nu_k)>0\}+\sum_{k=1}^m\kappa(\nu_k)$.
\label{lm:algogeneral}
\end{lem}

\begin{proof} The proof that $\lambda_{\kappa, \nu}$ is an edge-discriminator is very similar to the proof of Theorem \ref{th:hypergraph_main}. For $\kappa(\nu_k)=0$, we compare all pair of hyperedges $E_i, E_j \in \mathcal F$, such that $\nu( E_i, E_j)=\nu_k$. As all the vertices which are greater than $\nu_k$ are common to both or belongs to neither of the edges $E_i$ and $E_j$, according to the construction $\omega_\lambda(\underline{E}_{ij})-\omega_\lambda(\overline{E}_{ij})$ cannot change once $\lambda(\nu_k)$ is assigned. Hence, by the choice of $\lambda(\nu_k)$, $\omega_\lambda(E_i)\ne \omega_\lambda(E_j)$, from the $k$-th step onwards and therefore eventually, for any pair of edges such that $\nu( E_i, E_j)=\nu_k$ and $\kappa(\nu_k)=0$. 

When $\kappa(\nu_k)>0$, we define  $\lambda_{\kappa, \nu}(\nu_{k})=\kappa(\nu_k)+\inf\{\mathbb Z^+\cup\{0\}\backslash\mathcal B(\nu_k)\}$. This choice of $\lambda(\nu_k)$ differentiates any pair of edges $E_x, E_y \in \mathcal E$, with $\nu(E_x, E_y)=\nu_k$. Note that here we ignore the pairs $(E_z, \emptyset)$, where $\nu(E_z, \emptyset)=\nu_k$. However, as $\kappa(\nu_k)>0$, $\omega_{\lambda_{\kappa, \nu}}(E_z)>0$, for any edge $E_z$ such that $\nu(E_z, \emptyset)=\nu_k$. Hence, any pair of edges have distinct weights in this case as well.

Now, from the algorithm we can bound the weight of $\lambda_{\kappa, \nu}$ as follows:
\begin{eqnarray}
\sum_{k=1}^{m} \lambda_{\kappa, \nu}(\nu_{k}) &\leq&\sum_{k=1}^m |\mathcal A(\nu_k)|\pmb 1\{\kappa(\nu_k)=0\}+\sum_{k=1}^m |\mathcal B(\nu_k)|\pmb 1\{\kappa(\nu_k)>0\}
+\sum_{k=1}^m\kappa(\nu_k)\nonumber \\
&\leq&\sum_{k=1}^m |\mathcal B(\nu_k)|+\sum_{k=1}^m |\mathcal C(\nu_k)|\pmb 1\{\kappa(\nu_k)=0\}
+\sum_{k=1}^m\kappa(\nu_k)\nonumber \\
&\leq&\sum_{k=1}^m \chi(\nu_k)+\sum_{k=1}^m \pi(\nu_k)\pmb 1\{\kappa(\nu_k)=0\}
+\sum_{k=1}^m\kappa(\nu_k)\nonumber \\
&=&\frac{n(n-1)}{2}+\sum_{k=1}^m \pi(\nu_k)-\sum_{k=1}^m \pi(\nu_k)\pmb 1\{\kappa(\nu_k)>0\}
+\sum_{k=1}^m\kappa(\nu_k)\nonumber \\
&=&\frac{n(n+1)}{2}-\sum_{k=1}^m \pi(\nu_k)\pmb 1\{\kappa(\nu_k)>0\}+\sum_{k=1}^m\kappa(\nu_k).\nonumber
\label{eqn:n(n+1)/2_main_II}
\end{eqnarray}

\end{proof}

We will use the above lemma in completing the proof of Theorem \ref{th:hypergraph_main} and also in the proof of Theorem \ref{th:no_function}. 

\subsubsection{Completing the Proof of Theorem \ref{th:hypergraph_main}}


To complete the proof of Theorem \ref{th:hypergraph_main} we need to show that given any hypergraph 
$\mathcal H=(\mathcal V, \bbE)$, which contains two egdes $ E_x,  E_y\in \bbE$ such that $ E_x\cap  E_y\ne\emptyset$, it is possible to construct an edge-discriminating function $\lambda$ on $\mathcal V$ such that 
$\sum_{v\in \mathcal V}\lambda(v)< n(n+1)/2$. 

Let $|\mathcal V|=m$ and $v_o\in  E_x\cap  E_y$. Consider an ordering $\widetilde{\nu}: [m] \rightarrow \bbV$ such that $\widetilde\nu_m=\widetilde{\nu}(m)=v_o$. Then $v_o$ is the maximal vertex of both the edges $ E_x$ and $ E_y$, that is, $\pi(\widetilde\nu_m)\geq 2$. We start with the function $\kappa=\delta_{v_o}$ which takes the value $1$ at $v_o$ and $0$ everywhere else. Then from Lemma \ref{lm:algogeneral}, we know that there exists an edge-discriminator $\lambda_{\kappa, \widetilde{\nu}}$ such that $\sum_{v\in \mathcal V}\lambda_{\kappa, \widetilde{\nu}}(v)\leq n(n+1)/2-1$, which proves the tightness of the upper bound and completes the proof of Theorem \ref{th:hypergraph_main}.

\subsubsection{Consequences of Lemma \ref{lm:algogeneral}}
In this section  we discuss an immediate corollary of Lemma \ref{lm:algogeneral} which gives a slightly better upper bound on the weight of an edge-discriminator than in Theorem \ref{th:hypergraph_main}.

Given a collection of sets, a set which intersects all sets in the collection in at least one element is called a {\em hitting set}. Formally, for a hypergraph $\mathcal H=(\mathcal V, \mathcal E)$, a set $S\subseteq \mathcal V$ is called a hitting set of $\mathcal H$ if, for all edges $E\in \mathcal E$, $S\cap E\ne \emptyset$. A hitting set of the smallest size is called the {\em minimum hitting set} of $\mathcal H$. Hitting sets are also known as {\em vertex covers}, or more combinatorially {\em transversals}.

Using the above definition, we state the following immediate corollary of Lemma \ref{lm:algogeneral}:

\begin{cor}
For any hypergraph $\mathcal H=(\mathcal V, \mathcal E)$, $\omega_0(\mathcal H)\leq \frac{n(n-1)}{2}+\mathcal N(\mathcal H)$, where is the size of the minimum hitting set of $\mathcal N(\mathcal H)$.\label{corollary:upperimproved}
\end{cor}

\begin{proof}
For a fixed ordering $\nu$, denote by $N(\nu)$ the set of vertices $v\in \mathcal V$ with $\pi(v) >0$. Observe that $\min_\nu |N(\nu)|=\mathcal N(\mathcal H)$.

Now, consider any ordering $\nu$ on $\mathcal V$. Define $\kappa(v)=\pmb1\{\pi(v)>0\}$. Lemma \ref{lm:algogeneral} then implies that there exists an edge-discriminator $\lambda_{\kappa, \nu}$ such that $\sum_{v\in \mathcal V}\lambda_{\kappa, \nu}(v)\leq n(n-1)/2+|N(\nu)|$. The result follows by taking minimum over all orderings on $\mathcal V$. 
\end{proof}

\begin{remark} Note that the bound in Corollary \ref{corollary:upperimproved} is slightly better than the general upper bound proved in Theorem \ref{th:hypergraph_main}. Moreover, it is easy to see that this bound is attained by the optimal edge-discriminator of the star-graph $T_n$ on $n+1$ vertices and $n$ edges. This follows from the fact that the central vertex of $T_n$ is the minimum hitting set for $T_n$, and so $\mathcal N(\mathcal H)=1$.
\label{rm:nuattain}
\end{remark}

\section{Edge Discriminators in $r$-Uniform Hypergraphs}
\label{sec:sidon}

A $r$-uniform hypergraph is a hypergraph $\mathcal H=(\bbV, \bbE)$, where all the hyperedges in $\bbE$ have cardinality $r$. In other words, a $r$-uniform hypergraph is a set $\bbV$, and a collection $\bbE$ of $r$-element subsets of $\bbV$.

If $\bbE$ consists of the set of all $r$-element subsets of $\bbV$ then the hypergraph $\mathcal H=(\bbV, \bbE)$ is called the {\it complete $r$-uniform hypergraph} on $\bbV$, and is denoted by $\mathcal K^r_m$.

Consider an $r$-uniform hypergraph on $n$-vertices and let $\mathcal V=\{v_1, v_2, \ldots, v_m\}$. We can obtain an edge discriminator $\lambda: \bbV\rightarrow \mathbb Z^{+}\cup \{0\}$ if the sequence $\{\lambda(v_1), \lambda(v_2), \ldots, \lambda(v_m)\}$, has the property that sum of the elements of all its $r$-element subsets are mutually distinct. This property of the sequence $\lambda(v_1), \lambda(v_2), \ldots, \lambda(v_m)$ is closely related to the notion of Sidon sequences from additive number theory.

\subsection{Sidon Sequences and Their Generalizations}

A {\it Sidon sequence} is a sequence of natural numbers $A = \{a_1, a_2, \ldots \}$ such that all the pairwise sums $a_i + a_j$ $(i \leq j)$ are different. This problem was introduced by Sidon in 1932, during his investigations in Fourier analysis. The celebrated combinatorial problem, which asks for estimates of the maximum number of elements $s(m)$ from $\{1, 2, \ldots, m\}$ which form a Sidon sequence, was posed by Erd\H{o}s and Tur\' an \cite{erdosturansidon}. They proved that $s(m)\leq m^{1/2}+O(m^{1/4})$, which is the best possible upper bound except for the estimate of the error-term. The upper bound was refined by Lindstr\"om \cite{lindstrom} to $s(m)\leq m^{1/2}+m^{1/4}+1$ and further improved by Cilleruello \cite{cilleruellojcta}. A conjecture of Erd\H{o}s, with a 500 dollars reward attached to it, says that $s(m)=m^{1/2}+O(1)$.

Sidon-sequences can be generalized by considering sequences in which all $h$-element sums are mutually distinct. This leads to the following definition:

\begin{definition}
For a positive integer $h \geq 2$, a sequence of positive integers $A=\{a_1, a_2, \ldots\}$ is called a $B_h$-set if for every
positive integer $m$, the equation
$$m = a_1 + a_2 + \ldots + a_h, ~~~ a_1 \leq a_2 \leq \ldots, \leq a_h, ~~~ a_i \in  A,$$
has, at most, one solution. Let $F_h(m)$ denote cardinality of the largest $B_h$ set that can be selected
from the set $\{1, 2, \ldots, N\}$.
\end{definition}

Bose and Chowla \cite{bosechowla} proved that $F_h(m)\geq m^{1/h} + o(m^{1/h})$.
An easy counting argument implies that
$$F_h(m) \leq (hh!m)^{1/h}.$$

The above upper bound has gone through many improvements and refinements. The general upper bound has the form $F_h(m)\leq c(h)m^{1/h}+o(m^{1/h})$, where $c(h)$ is a constant depending on $h$. For specific values of $c(h)$ refer to Cilleruello \cite{cilleruello_upper}, Cilleruello and Jimenez \cite{cilleruello_mathematik}, Jia \cite{jiasidon}, and the references therein. Other related results and problems on Sidon sequences and $B_h$-sets can be found in the surveys \cite{sequences_book,sidonsurvey}.

\subsection{Proof of Proposition \ref{th:r+1}}

In this section we shall use the notion of $B_h$-sets to obtain new bounds on the weight of an edge-discriminator
in $r$-uniform hypergraphs. Let $G_h(m)$ be the minimum of the maximum element taken over all $B_h$-sets of length $m$. In other words, $G_h(m)$ is the inverse function of $F_h(m)$, and a lower bound for $F_h(m)$ corresponds to an  upper bound for $G_h(m)$.

Consider a $r$-uniform hypergraph $\mathcal H=(\mathcal V, \mathcal E)$, with $\mathcal V=\{v_1, v_2, \ldots, v_m\}$. Note that a function $\lambda: \mathcal V\rightarrow \mathbb Z^+\cup\{0\}$ such that $\{\lambda(v_1), \lambda(v_2), \ldots, \lambda(v_m)\}$ is a $B_r$-set is a edge-discriminator for $\mathcal H$. Now, as $G_r(m)\leq m^{r}+o(m^r)$, we have

Therefore, $$\omega_\lambda(\mathcal H)=\sum_{i=1}^m\lambda(v_i)\leq m\cdot\max_{1\leq i \leq m}\lambda(v_i)\leq m^{r+1}+o(m^{r+1}).$$

In case of the complete $r$-uniform hypergraph $\mathcal K^r_m=(\mathcal V, \mathcal E)$, we have $|\bbE|={{m\choose{r}}}$ and every vertex $v\in \mathcal V$ belongs to ${{m-1}\choose{r-1}}$ hyperedges.
This implies that any edge-discriminator $\lambda: \mathcal V\rightarrow \mathbb Z^+\cup\{0\}$ on $\mathcal K_r(\mathcal V)$ satisfy
$$\sum_{E\in \mathcal E}\omega_\lambda(E)=\sum_{E\in \mathcal E}\sum_{v\in E}\lambda(v)={m-1\choose{r-1}}\sum_{v\in \mathcal V}\lambda(v).$$
As $\omega_\lambda(E)$ is distinct for all $E\in \bbE$, we have $\sum_{E\in \mathcal E}\omega_\lambda(E) \geq |\mathcal E|(|\mathcal E|+1)/2$. Therefore,
$$\omega_\lambda(\mathcal K^r_m)=\sum_{v\in \mathcal V}\lambda(v)=\frac{1}{{m-1\choose{r-1}}}\sum_{E\in \mathcal E}\omega_\lambda(E) \geq \frac{{{m\choose{r}}}\left({{m\choose{r}}}+1\right)}{2{{m-1\choose{r-1}}}}
=\frac{\frac{m}{r}\left({{m\choose{r}}}+1\right)}{2}\geq c m^{r+1},$$
for $m$ large enough and some constant $c$. This proves that the weight of the optimal edge-discriminator of $\mathcal K^r_m$ is within a constant factor of the upper bound, and completes the proof of Proposition \ref{th:r+1}.

\begin{remark}The upper bound on the weight of an edge-discriminator for a $r$-uniform hypergraph obtained in Proposition \ref{th:r+1} is often better than the general $|\mathcal E|(|\mathcal E|+1)/2=O(|\mathcal E|^2)$ upper bound proved in Theorem \ref{th:hypergraph_main}. In particular, when $|\mathcal E|\geq c|\mathcal V|^{\frac{r+1}{2}}$, then Proposition \ref{th:r+1} provides a better bound on the weight of an edge-discriminator. For example, if the $r$-uniform hypergraph is dense then Proposition \ref{th:r+1} gives a sharper upper bound. On the other hand, if we have a sparse $r$-uniform hypergraph, then Theorem \ref{th:hypergraph_main} gives a better upper bound of the weight of the edge-discriminator.
\end{remark}

\section{Lower Bound on the Weight of an Edge-Discriminator}
\label{sec:lowerbound}

In this section we prove a simple lower bound on the weight of an edge discriminator on $\mathcal H$. Given a hypergraph $\mathcal H=(\mathcal V, \mathcal E)$, subset $\mathcal E'\subseteq \mathcal E$ is said to be a {\it matching} if the elements in $\mathcal E'$ are mutually disjoint. A {\it maximum matching} of $\mathcal H$ is the the largest size matching in $\mathcal H$.

\begin{thm}For any edge-discriminator $\lambda$ on a hypergraph $\mathcal H=(\mathcal V, \mathcal E)$, with $|\mathcal E|=n$, $\omega_\lambda(\mathcal V)\geq \max\{n, \frac{\delta(\delta+1)}{2}\}$, where $\delta$ is the size of the maximum matching of $\mathcal H$. Moreover, there is a hypergraph with $n$ edges which attains this bound.
\label{th:lower_bound}
\end{thm}

\begin{proof}Observe that the weights $\omega_\lambda(E_1), \omega_\lambda( E_2), \ldots, \omega_\lambda( E_n)$ are all positive and distinct. This implies that
\begin{equation}
\omega_\lambda(\mathcal V)\geq \max_{ E_i\in \bbE}{\omega_\lambda( E_i)}\geq n.
\label{eq:lowerbound_I}
\end{equation}

Next, suppose that $ \mathcal E'=\{ E_{i_1},  E_{i_2}, \ldots,  E_{i_{\delta}}\}$ is a matching of $\mathcal H$. Since the hyperedges $ E_{i_1},  E_{i_2}, \ldots,  E_{i_{\delta}}$ are mutually disjoint and the corresponding weights $\omega_\lambda( E_{i_1}), \omega_\lambda( E_{i_2}), \ldots, \omega_\lambda( E_{i_{\delta}})$ are distinct, we have,
\begin{equation}
\omega_\lambda(\mathcal V)\geq \sum_{j=1}^{\delta}\omega_\lambda( E_{i_j})\geq \frac{\delta(\delta+1)}{2}.
\label{eq:lowerbound_II}
\end{equation}
The result now follows by combining Equation (\ref{eq:lowerbound_I}) and Equation (\ref{eq:lowerbound_II}). 

Finally, consider the hypergraph $\mathcal H_O=(\bbV_O, \bbE_O)$, with $\bbV_O=[n]$ and $\bbE_O=\{E_1, E_2, \ldots, E_n\}$, where $E_i=[i]$. Then it is easy to see that the function $\lambda_O:\mathcal V_O \rightarrow \mathbb Z^+\cup \{0\}$ defined by $\lambda_O(i)=1$, for all $i \in [n]$, is the optimal edge-discrimiator for $\mathcal H_O$ and 
$\omega_0(\mathcal H_O)=\omega_{\lambda_O}(\mathcal V)=n$ attains the lower bound.
\end{proof}

\begin{remark}
A hypergraph is called $d$-{\it regular} if every vertex is in $d$ hyperedges. Let $\mathcal H=(\mathcal V, \mathcal E)$ be a $d$-regular hypergraph with $|\mathcal E|=n$. Then any edge discriminator $\lambda$ on $\mathcal H$ satisfies $\omega_\lambda(\mathcal V)=\sum_{v\in \mathcal V}\lambda(v)=\frac{1}{d}\sum_{E\in \mathcal E}\omega_\lambda(E)\geq \frac{n(n+1)}{2d}=O(n^2/d)$. We shall see later that in many $d$-regular hypergraphs the weight of the optimal edge-discriminator is within a constant factor of this lower bound. In fact, if $d=O(1)$, that is, the hypergraph is sparse, then this lower bound has the same of order of magnitude as the $O(n^2)$ upper bound proved in Theorem \ref{th:hypergraph_main}.
\end{remark}

\section{Optimal Edge-Discriminators in Special Hypergraphs}
\label{sec:computation}

In this section we shall compute the optimal edge discriminators for some special hypergraphs. The section is divided into four subsections where we compute the optimal edge-discriminator for the power set hypergraph, paths, cycles and the complete $r$-partite hypergraph, respectively.

\subsection{Power Set Hypergraph}
\label{sec:power_set}

The power set hypergraph on a set $\mathcal V$, with $|\mathcal V|=m$, is the hypergraph $(\mathcal V, 2^\mathcal V)$, where $2^\mathcal V$ denotes the power set of $\mathcal V$, that is, the set of all non-empty subsets of $\mathcal V$. We denote the power set hypergraph on a set of $m$ elements by $Power(m)$.

\begin{thm}$\omega_0(Power(m))=2^m-1$.
\label{th:power_set}
\end{thm}

\begin{proof}Let $\mathcal V=\{v_1, v_2, \ldots, v_m\}$. As $|2^\mathcal V|=2^m-1$, by Theorem \ref{th:lower_bound} it suffices to construct an edge-discriminator $\lambda$ on $(\mathcal V, 2^\mathcal V)$ such that $\omega_\lambda((\mathcal V, 2^\mathcal V))=2^m-1$.

Define $\lambda: \mathcal V \rightarrow \mathbb Z^+\cup \{0\}$ as: $\lambda(v_i)=2^{i-1}$, for $i\in [m]$. For any subset $\mathcal W=\{v_{i_1}, v_{i_2}, \ldots, v_{i_k}\}$, with $1\leq i_1 < i_2 < \ldots < i_k \leq m$, we have $\omega_\lambda(\mathcal W)=\sum_{i=1}^m q_i(\mathcal W)2^{i-1}$, where $q_i(\mathcal W)$ is 1 or 0 depending on whether $i\in \{i_1, i_2, \ldots, i_k\}$ or not. As for any two distinct subsets $\mathcal W$ and $\mathcal Y$ the $m$-tuples $(q_1(\mathcal W), q_2(\mathcal W), \ldots, q_m(\mathcal W))$ and $(q_1(\mathcal Y), q_2(\mathcal Y), \ldots, q_m(\mathcal Y))$ are distinct, we have $\omega_\lambda(\mathcal W)\ne \omega_\lambda(\mathcal Y)$. Therefore, $\lambda$ is an edge-discriminator with $\omega_\lambda(\mathcal V, 2^\mathcal V))=\sum_{i=1}^m 2^{i-1}=2^m-1$, and the result follows. 
\end{proof}

\subsection{Paths: Proof of Theorem \ref{th:path}}
\label{subsection:path}

Let $P_m=(V, E)$ be the path on $m$ vertices. Suppose that $V=\{v_1, v_2, \ldots, v_m\}$, where the vertices are taken in order from left to right. The vertices $v_1$ and $v_m$ are called the end-vertices of the path $P_m$. Denote by $e_i$ the edge $(v_i, v_{i+1})$, for $i \in [m-1]$.

Suppose $\lambda$ is an edge-discriminator on $P_m$. Then as each of the vertices $v_2, v_3, \ldots, v_{m-1}$ are incident on 2 edges, we get
$$\sum_{i=1}^{m-1}\omega_\lambda(e_i)=\lambda(v_1)+\lambda(v_m)+2\sum_{i=2}^{m-1}\lambda(v_i).$$
As $\omega_\lambda(e_i)$ is distinct for all $i \in [m-1]$, therefore,
\begin{equation}
\omega_\lambda(P_m)=\sum_{i=1}^n \lambda(v_i)\geq \frac{\lambda(v_1)+\lambda(v_m)}{2}+\sum_{i=2}^{m-1}\lambda(v_i)=\frac{1}{2}\cdot\sum_{i=1}^{m-1}\omega_\lambda(e_i)\geq \frac{m(m-1)}{4}.
\label{eq:path}
\end{equation}

\begin{figure*}[h]
\centering
\begin{minipage}[c]{1.0\textwidth}
\centering
\includegraphics[width=5.25in]
    {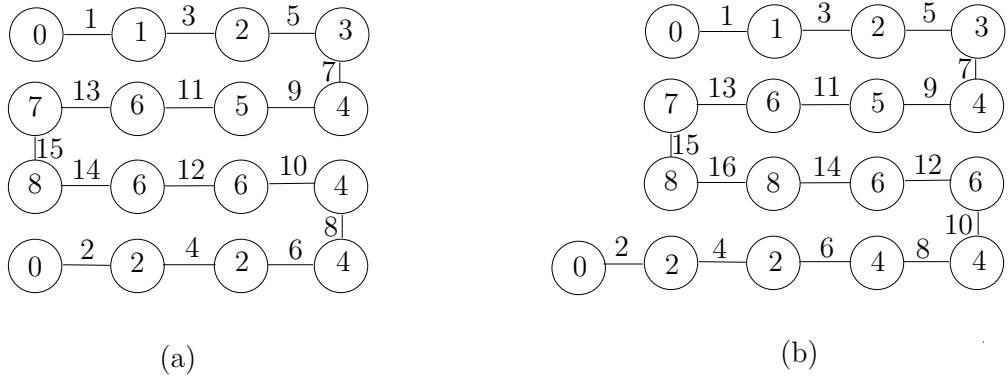}\\
\end{minipage}%
\caption{Optimal edge-discrimination for $P_m$: (a) $m$ is a multiple of 4, and (b) $m$ is 1 modulo 4.}
\label{fig:path_m_1}
\end{figure*}

We now have the following two cases:

\begin{description}
\item[{\it Case} 1:]$m=4k$ or $m=4k+1$ for some integer $k \geq 1$. In this case $m(m-1)/4$ is an integer, and from Equation (\ref{eq:path}) any edge-discriminator of $P_m$ with weight $m(m-1)/4$ will  be optimal. Note that equality holds in Equation (\ref{eq:path}) if it is possible to find an edge-discriminator $\lambda'$ on $P_m$ such that $\lambda'(v_1)=\lambda'(v_m)=0$ and the set $\{\omega_{\lambda'}(e_i): 1\leq i \leq m-1\}$ takes all values in $[m-1]$.
\begin{description}
\item[{\it Case} 1.1:]$m=4k$. Consider the function $\lambda': V\rightarrow \mathbb Z^+\cup\{0\}$ as follows:
$$\begin{array}{ll}
    \lambda'(v_i)=i-1, & \hbox{for}~i\in [m/2+1]; \\
    \lambda'(v_{2i+1})=\lambda'(v_{2i})=m-2i, & \hbox{for}~ i\in [m/4+1, m/2-1];  \\
    \lambda'(v_m)=0. &
  \end{array}$$
Then $\lambda'(v_1)=\lambda'(v_m)=0$. Now, for $e_i=(v_i, v_{i+1})$, with $i\leq m/2$, we have $\omega_{\lambda'}(e_i)=2i-1$, which implies that $F=\{\omega_{\lambda'}(e_i)| i \in [m/2]\}=\{2i-1|i \in [m/2]\}$ is the set of all odd numbers less than $m$. Note that $\omega_{\lambda'}(e_{m/2+1})=m-2$, and
for $i\geq m/2+2$, $B=\{\omega_{\lambda'}(e_i)| i \in [m/2+1, m-1]\}$ is the set of all even numbers less than $m-2$. Therefore, the set $\{\omega_{\lambda'}(e_i): i \in [m-1]\}$ takes all values in $[m-1]$. This implies that $\lambda'=\lambda_{P_m}$ and $\omega_0(P_m)=\omega_{\lambda'}(P_m)=m(m-1)/4$. The construction of the optimal edge discriminator for a path with 16 vertices is shown in Figure \ref{fig:path_m_1}(a).

\item[{\it Case} 1.2:]$m=4k+1$. In this case consider the function $\lambda': V\rightarrow \mathbb Z^+\cup\{0\}$ as follows:
$$\begin{array}{ll}
    \lambda'(v_i)=i-1, & \hbox{for}~i\in [(m-1)/2]; \\
    \lambda'(v_{2i+1})=\lambda'(v_{2i})=m-2i, & \hbox{for}~ i\in \{\frac{m+1}{4}+j| 0\leq j \leq \frac{m-5}{4}\};  \\
    \lambda'(v_m)=0. &
  \end{array}$$
Now, from arguments exactly similar to the previous case it follows, $\omega_{\lambda_{P_m}}(P_m)=\omega_{\lambda'}(P_m)=n(n-1)/4$. The construction of the optimal edge discriminator for a path with 17 vertices is shown in Figure \ref{fig:path_m_1}(b).
\end{description}

\item[{\it Case} 2:]$m=4k+2$ or $m=4k+3$ for some positive integer $k$. In this case $m(m-1)/4$ is not an integer. Suppose that there exists an edge-discriminator $\lambda'$ on $P_m$ such that $\{\omega_{\lambda'}(e_i): 1\leq i \leq m-1\}$ takes all the values in $[m-1]$. Then, from Equation (\ref{eq:path}) we get
    $$\omega_\lambda'(P_m)\geq
\frac{\lambda'(v_1)+\lambda'(v_m)}{2}+\sum_{i=2}^{m-1}\lambda'(v_i)=\frac{m(m-1)}{4}.$$
    Now, as $\omega_\lambda'(P_m)$ is an integer and is $m(m-1)/4$ is not an integer, we must have $\omega_\lambda'(P_m)\geq m(m-1)/4+1/2=\lceil m(m-1)/4\rceil$ and equality holds if we have $\lambda'(v_1)=0$ and $\lambda'(v_m)=1$.
\begin{description}
\item[{\it Case} 2.1:]$m=4k+2$. Consider the function $\lambda': V\rightarrow \mathbb Z^+\cup\{0\}$ as follows:
$$\begin{array}{ll}
    \lambda'(v_i)=i-1, & \hbox{for}~i\in [m/2+1]; \\
    \lambda'(v_{2i-1})=\lambda(v_{2i})=m-2i+1, & \hbox{for}~ i\in [m/4+3/2, m/2].
  \end{array}$$
It is easy to see that $\lambda'(v_1)=0$ and $\lambda'(v_m)=1$, and $\{\omega_{\lambda'}(e_i): 1\leq i \leq m-1\}$ takes all the values in $[n-1]$. Therefore, $\omega_0(P_m)=\omega_{\lambda'}(P_m)=m(m-1)/4+1/2=\lceil m(m-1)/4\rceil$. The construction of the optimal edge discriminator for a path with 18 vertices is shown in Figure \ref{fig:path_m_2}(a).

\begin{figure*}[h]
\centering
\begin{minipage}[c]{1.0\textwidth}
\centering
\includegraphics[width=5.25in]
    {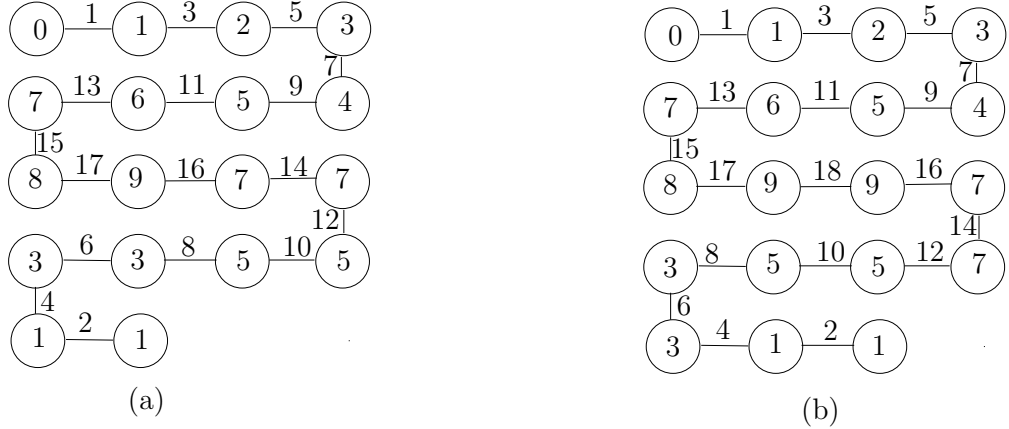}\\
\end{minipage}%
\caption{Optimal edge-discrimination for $P_m$: (a) $m$ is 2 modulo 4, and (b) $m$ is 3 modulo 4.}
\label{fig:path_m_2}
\end{figure*}

\item[{\it Case} 2.2:]$m=4k+3$. In this case, the optimal edge-discriminator function $\lambda': V\rightarrow \mathbb Z^+\cup\{0\}$ is:
$$\begin{array}{ll}
    \lambda'(v_i)=i-1, & \hbox{for}~i\in [(m-1)/2]; \\
    \lambda'(v_{2i-1})=\lambda(v_{2i})=m-2i+1, & \hbox{for}~ i\in \{\frac{m+3}{4}+j| 0 \leq j \leq \frac{m-3}{4}\}.
  \end{array}$$
The construction of the optimal edge discriminator for a path with 19 vertices is shown in Figure \ref{fig:path_m_2}(b).

\end{description}
\end{description}

\subsection{Cycles: Proof of Theorem \ref{th:cycle}}

Let $C_m=(V, E)$ be the cycle on $m$ vertices. Suppose that $V=\{v_1, v_2, \ldots, v_m\}$, with the vertices taken in the clockwise order. Denote by $e_i$ the edge $(v_i, v_{i+1})$, for $i \in [m]$. Suppose $\lambda$ is an edge-discriminator on $C_m$. As every vertex is incident to exactly two edges, we get $\sum_{i=1}^{m}\omega_\lambda(e_i)=2\sum_{i=1}^{m}\lambda(v_i)$.
Since $\omega_\lambda(e_i)$ is distinct for all $i \in  [m]$,
\begin{equation}
\omega_\lambda(C_m)=\sum_{i=1}^m \lambda(v_i) = \frac{1}{2}\cdot\sum_{i=1}^{m}\omega_\lambda(e_i)\geq \frac{m(m+1)}{4}.
\label{eq:cycle}
\end{equation}

Before we find the optimal edge-discriminator of $C_m$ we prove the following lemma:

\begin{lem}Suppose $P_m=(V, E)$ be a path, where $m$ is either 0 or 1 modulo 4. Let $V=\{v_1, v_2, \ldots, v_m\}$, with the vertices taken from left to right, and $e_i$ be the edge $(v_i, v_{i+1})$. Then there exists an edge discriminator $\lambda$ on $P_m$ such that $\lambda(v_1)=0$, $\lambda(v_m)=2$, and the set $\{\omega_\lambda(e_i)| 1\leq i \leq m-1\}$ takes all the values in $[m-1]$, and $\omega_{\lambda}(P_m)=m(m-1)/4+1$.
\label{lm:path02}
\end{lem}

\begin{proof}
Assume that $m=4k$ or $4k+1$ for some integer $k>0$. For the case $k=1$, $k$ is either 4 or 5 and in each of these cases we can construct an edge-discriminator $\lambda$ as follows:

\begin{description}
\item[{\it Case} 1:]$m=4$. Define $\lambda(v_1)=0$, $\lambda(v_2)=\lambda(v_3)=1$ and $\lambda(v_3)=2$. In this case $\omega_\lambda(P_m)=4$ and $\lambda$ satisfies the required properties.

\item[{\it Case} 2:]$m=5$. Define $\lambda(v_1)=0$, $\lambda(v_2)=\lambda(v_3)=1$ and $\lambda(v_3)=\lambda(v_4)=2$. Then $\omega_\lambda(P_m)=6$ and $\lambda$ is an edge-discriminator satisfying the required properties.
\end{description}
Therefore, suppose $k\geq 2$, which implies that $m\geq 8$. Let $V'=V\backslash\{v_{m-3}, v_{m-2}, v_{m-1}, v_m\}$ and $E'=E\backslash\{e_{m-3}, e_{m-2}, e_{m-1}\}$. From {\it Case} 1 of the Section \ref{subsection:path} we know that there is edge-discriminator $\lambda'$ on $P_{m-4}=(V', E')$ such that $\lambda'(v_1)=\lambda'(v_{m-4})=0$, the set $\{\omega_{\lambda'}(e_i)| 1\leq i \leq m-5\}$ takes all the values in $[m-5]$, and $\omega_{\lambda'}(P_m)=\frac{(m-5)(m-4)}{4}$. We now extend $\lambda'$ to an edge-discriminator $\lambda$ on $P_m=(V, E)$ as follows:
$$\lambda(v_i)=\left\{
\begin{array}{lll}
    \lambda'(v_i), & \hbox{if} & i\in [m-4]; \\
    m-4, & \hbox{if} & i=m-3;\\
    1, & \hbox{if} & i=m-2;\\
    m-3, & \hbox{if} & i=m-1;\\
    2, & \hbox{if} & i=m.
  \end{array}\right.$$

Clearly, $\lambda$ is an edge-discriminator on $P_m$ with edge weights in $[m-1]$ and
$\omega_\lambda(P_m)=\frac{(m-5)(m-4)}{2}+2m-4=\frac{m(m-1)}{4}+1$. 
\end{proof}

Using the above lemma we now construct an optimal edge-discriminator for $C_m$. Consider the following two cases:
\begin{description}
\item[{\it Case} 1:]$m$ is 0 or 3 modulo 4. In this case $m(m+1)/4$ is an integer. Therefore, equality holds in Equation (\ref{eq:cycle}) if it is possible to find an edge-discriminator $\lambda$ on $C_m$ such that the set $\{\omega_{\lambda}(e_i): 1\leq i \leq m\}$ takes all values in $[m]$. Note that $m+1$ is either 0 or 1 modulo 4, and by {\it Case} 1 of Section \ref{subsection:path}, we know that there exists an edge discriminator $\lambda'$ on $P_{m+1}$ such that the weights of the $m$ edges of $P_{m+1}$ takes all the values in $[m]$. Moreover, as $\lambda'$ assigns value 0 to the two end vertices of $P_{m+1}$ amalgamating the two end vertices of $P_{m+1}$ we get an edge-discriminator $\lambda$ on $C_m$ such that $\{\omega_{\lambda}(e_i): 1\leq i \leq m\}$ takes all values in $[m]$. This proves that $\lambda=\lambda_{C_m}$ and $\omega_0(C_m)=m(m+1)/4$.

\item[{\it Case} 2:]$m$ is 1 or 2 modulo 4. In this case $m(m+1)/4$ is not an integer and so
$\omega_\lambda(C_m)\geq \frac{m(m+1)}{4}+\frac{1}{2}$. Consider the path $P_{m-1}=(V', E')$, $V'=\{v_1, v_2, \ldots, v_{m-1}\}$ with the vertices taken from left to right. Since $m-1$ is 0 or 1 modulo 4, by Lemma \ref{lm:path02} there exists an edge-discriminator $\lambda'$  on $P_{m-1}$ such that $\lambda'(v_1)=0$, $\lambda'(v_{m-1})=2$, and the set $\{\omega_{\lambda'}(e_i)| 1\leq i \leq m-2\}$ takes all the values in $[m-2]$, and $\omega_{\lambda'}(P_{m-1})=\frac{(m-2)(m-1)}{4}+1$. Now, we extend $\lambda'$ to an edge discriminator $\lambda''$ on $P_{m+1}=(V, E)$, with $V=\{v_1, v_2, \ldots, v_{m+1}\}$ as follows:
$$\lambda''(v_i)=\left\{
\begin{array}{lll}
    \lambda'(v_i), & \hbox{if} & i\in [m-1]; \\
    m-1, & \hbox{if} & i=m;\\
    0, & \hbox{if} & i=m+1.\\
\end{array}\right.$$
Note, $\lambda''$ assigns value 0 to the two end vertices of $P_{m+1}$. Thus, amalgamating the two end vertices of $P_{m+1}$ we get an edge-discriminator $\lambda$ on $C_m$ such that $\omega_{\lambda}(C_m)=\frac{(m-2)(m-1)}{4} + 1 + (m-1)=\frac{m(m+1)}{4}+\frac{1}{2}$. This implies that $\lambda=\lambda_{C_m}$ and the result follows.
\end{description}

\subsection{Complete $r$-Partite Hypergraph: Proof of Theorem \ref{th:r-partite}}

In this section we shall construct the optimal edge-discriminator for complete $r$-partite hypergraph $\mathcal H_r(\pmb a)=(\mathcal V_r, \bbE_r)$, where $\pmb a=(m_1, m_2, \ldots, m_r)$. Let $\lambda: \mathcal V_r \rightarrow \mathbb Z^+\cup \{0\}$ be an edge discriminator for the hypergraph $\mathcal H_r(\pmb a)$. Define $\omega_\lambda(A_i)=\sum_{v\in A_i}\lambda(v)$, for $i \in [r]$. For $q\leq  r$, the hypergraph with vertex set $\mathcal V_q=\bigcup_{i=1}^q A_{i}$ and hyperedge-set $\bbE_q=A_{1}\times A_{2}\times \ldots \times A_{q}$ is to be denoted by $\mathcal H_r^q(\pmb a)$. Let $\lambda_q$ denote the restriction of $\lambda$ to $\mathcal V_q$, define $\omega_\lambda(\bbE_q):=\sum_{e\in  \bbE_q}\omega_{\lambda_q}(e)$ for $q\leq r$.

Note that a vertex $v\in A_i$, $i\in [r]$, belongs to $m_{(r)}/m_i$ many hyperedges, where $m_{(q)}=\prod_{k=1}^q m_k$ for any positive integer $q\leq r$. As $| \bbE_r|=m_{(r)}$, we have the following equality,
\begin{equation}
\omega_\lambda(\bbE_r)=\sum_{e\in  \bbE_r}\omega_\lambda(e)=\sum_{i=1}^r\frac{m_{(r)}}{m_i}\sum_{v\in A_i}\lambda(v)=\sum_{i=1}^r\frac{m_{(r)}}{m_i}\omega_\lambda(A_i).
\end{equation}
This implies that
\begin{eqnarray}
\sum_{i=1}^r\omega_\lambda(A_i)&=&\omega_{\lambda}(A_r)+\sum_{i=1}^{r-1}\left(\frac{m_r}{m_i}\right)\omega_\lambda(A_i)+
\sum_{i=1}^{r-1}\left(1-\frac{m_r}{m_i}\right)\omega_\lambda(A_i)\nonumber\\
&=&\left(\frac{m_r}{m_{(r)}}\right)\omega_\lambda(\bbE_r)+
\sum_{i=1}^{r-1}\left(1-\frac{m_r}{m_i}\right)\omega_\lambda(A_i).
\label{eqn:r-partite1}
\end{eqnarray}
Consider the restriction of $\lambda_{r-1}$ of $\lambda$ to $\mathcal H_r^{r-1}(\pmb a)=(\bbV_{r-1}, \bbE_{r-1})$. Note that vertex $v\in A_i$, $i\in [r-1]$, belongs to $m_{(r-1)}/m_i$ many hyperedges in the hypergraph $\mathcal H_r^{r-1}(\pmb a)$. Thus, we have
\begin{equation}
\omega_{\lambda}(\bbE_{r-1}):=\sum_{e\in  \bbE_{r-1}}\omega_{\lambda_{r-1}}(e)=\sum_{i=1}^{r-1}\frac{m_{(r-1)}}{m_i}\sum_{v\in A_i}\lambda(v).
\label{eqn:r-1}
\end{equation}
Therefore,
\begin{eqnarray}
\sum_{i=1}^{r-1}\left(\frac{1-\frac{m_r}{m_i}}{1-\frac{m_r}{m_{r-1}}}\right)\omega_\lambda(A_i)
&=&\omega_{\lambda}(A_{r-1})+\sum_{i=1}^{r-2}\left(\frac{m_{r-1}}{m_{i}}\right)\omega_\lambda(A_i)
+\sum_{i=1}^{r-2}\left(\frac{1-\frac{m_r}{m_i}}{1-\frac{m_r}{m_{r-1}}}-\frac{m_{r-1}}{m_{i}}\right)\omega_\lambda(A_i)\nonumber\\
&=&\left(\frac{m_{r-1}}{m_{(r-1)}}\right)\omega_\lambda(\bbE_{r-1}) +\sum_{i=1}^{r-2}\left(\frac{m_{r-1}(m_i-m_{r-1})}{m_i(m_{r-1}-m_r)}\right)\omega_\lambda(A_i)
\label{eqn:middle}
\end{eqnarray}
where the last equality follows from Equation (\ref{eqn:r-1}). Multiplying both sides of Equation (\ref{eqn:middle}) by $1-\frac{m_r}{m_{r-1}}$ we get
\begin{eqnarray}
\sum_{i=1}^{r-1}\left(1-\frac{m_r}{m_i}\right)\omega_\lambda(A_i)
=\left(\frac{m_{r-1}-m_r}{m_{(r-1)}}\right)\omega_\lambda(\bbE_{r-1}) +\sum_{i=1}^{r-2}
\left(1- \frac{m_{r-1}}{m_i}\right)\omega_\lambda(A_i).
\label{eqn:r-partite2}
\end{eqnarray}
Now, substituting Equation (\ref{eqn:r-partite2}) in Equation (\ref{eqn:r-partite1}) we get,
\begin{eqnarray}
\sum_{i=1}^r\omega_\lambda(A_i)=\left(\frac{m_r}{m_{(r)}}\right)\omega_\lambda(\mathcal E)+\left(\frac{m_{r-1}-m_r}{m_{(r-1)}}\right)\omega_\lambda(\bbE_{r-1}) +\sum_{i=1}^{r-2}
\left(1- \frac{m_{r-1}}{m_i}\right)\omega_\lambda(A_i).
\end{eqnarray}
Proceeding in this way we ultimately get,
\begin{eqnarray}
\sum_{i=1}^r\omega_\lambda(A_i)&=&\left(\frac{m_r}{m_{(r)}}\right)\omega_\lambda(\mathcal E)+\sum_{q=2}^r \left(\frac{m_{q-1}-m_q}{m_{(q-1)}}\right)\omega_\lambda(\bbE_{q-1})
\label{eqn:final}
\end{eqnarray}
Note that the weights $\omega_{\lambda_q}(e)$ for $e\in \bbE_q$ are all distinct, for $q < r$. Therefore, $\{\omega_{\lambda_q}(e)|e\in \bbE_q\}$, consists of $m_{(q)}$ non-negative distinct values for $q < r$. This implies that $\omega_\lambda(\bbE_q)=\sum_{e\in \bbE_q}\omega_{\lambda_q}(e)\geq m_{(q)}(m_{(q)}-1)/2$. Moreover, as $\lambda$ is an edge-discriminator for $\mathcal H_r(\pmb a)$, $\{\omega_\lambda(e)|e\in \mathcal E\}$, consists of $m_{(r)}$ distinct positive values. This implies that $\omega_\lambda(\bbE)=\sum_{e\in \bbE}\lambda(e)\geq m_{(r)}(m_{(r)}+1)/2$. Thus, from Equation (\ref{eqn:final}) we get
\begin{eqnarray}
\sum_{i=1}^r\omega_\lambda(A_i)&\geq&\frac{1}{2}\left[m_r(m_{(r)}+1)+\sum_{q=2}^{r}(m_{q-1}-m_q)(m_{(q-1)}-1)\right]\nonumber\\
&=&m_r+\sum_{q=1}^r\frac{m_{(q)}(m_q-1)}{2}.
\label{eqn:partite}
\end{eqnarray}

We shall now construct an edge discriminator $\lambda_{\mathrm{OPT}}$ on $\mathcal H_r(\pmb a)$ whose weight attains the above bound. Denote the elements of $\bigcup_{i=1}^r A_i$ as follows: $A_i=\{v_{ij}| j\in[m_i]\}$, $i\in [r]$.
Let $\lambda_{\mathrm{OPT}}(A_i)$ denote the vector $(\lambda_{\mathrm{OPT}}(v_{i1}), \lambda_{\mathrm{OPT}}(v_{i2}), \ldots, \lambda_{\mathrm{OPT}}(v_{im_i}))$, for $i\in [r]$. Define $\lambda_{\mathrm{OPT}}$ as follows:
\begin{eqnarray}
\lambda_{\mathrm{OPT}}(A_1)&=&(0, 1, 2, \ldots, m_1-1)\nonumber\\
\lambda_{\mathrm{OPT}}(A_2)&=&(0, m_{(1)}, 2m_{(1)}, \ldots, (m_2-1)m_{(1)})\nonumber\\
&\vdots&\nonumber\\
\lambda_{\mathrm{OPT}}(A_{r-1})&=&(0, m_{(r-2)}, 2m_{(r-2)}, \ldots, (m_{r-1}-1)m_{(r-2)})\nonumber\\
\lambda_{\mathrm{OPT}}(A_r)&=&(1, m_{(r-1)}+1, 2m_{(r-1)}+1, \ldots, (m_r-1)m_{(r-1)}+1).
\end{eqnarray}

It is easy to see that $\omega_{\lambda_{\mathrm{OPT}}}(\mathcal V)=\sum_{i=1}^r\omega_{\lambda_{\mathrm{OPT}}}(A_i)=m_r+\sum_{q=1}^r\frac{m_{(q)}(m_q-1)}{2}$.

We now check that $\lambda_{\mathrm{OPT}}$ is an edge-discriminator. Define $m_{(0)}=1$. A hyperedge $e\in \mathcal E$ is of the form $(v_{1i_1}, v_{2i_2}, \ldots, v_{ri_r}\}$, where $1\leq i_k \leq m_k$ for $k\in [r]$. Therefore, a hyperedge in $\mathcal E$ is uniquely determined by the $r$-tuple $(i_1, i_2, \ldots, i_r)$, and we denote it by $e(i_1, i_2, \ldots, i_r)$. The weight of hyperedge $e(i_1, i_2, \ldots, i_r)$ is
$$\omega_{\lambda_{\mathrm{OPT}}}(e(i_1, i_2, \ldots, i_r))=(i_1-1)+(i_2-1)m_{(1)}+\ldots +(i_{r-1}-1)m_{(r-2)}+ (i_r-1)m_{(r-1)}+1.$$
Now, if possible, suppose that for two $r$-tuples $(i_1, i_2, \ldots, i_r)$ and $(j_1, j_2, \ldots, j_r)$, with $1\leq i_k, j_k \leq m_k $, we have $\omega_{\lambda_{\mathrm{OPT}}}(e(i_1, i_2, \ldots, i_r))=\omega_{\lambda_{\mathrm{OPT}}}(e(j_1, j_2, \ldots, j_r))$. Let $q\leq r$ be the largest index such that $i_q\ne j_q$. W.l.o.g. assume that $i_q > j_q$. Then $\omega_{\lambda_{\mathrm{OPT}}}(e(i_1, i_2, \ldots, i_r))=\omega_{\lambda_{\mathrm{OPT}}}(e(j_1, j_2, \ldots, j_r))$ implies that $$(i_q-j_q)m_{(q-1)}=\left|\sum_{k=1}^{q-1}(j_k-i_k)m_{(k-1)}\right|\leq \sum_{k=1}^{q-1}|j_k-i_k|m_{(k-1)}\leq \sum_{k=1}^{q-1}(m_k-1)m_{(k-1)}=m_{(q-1)}-1,$$
which is impossible.

This proves that $\lambda_{\mathrm{OPT}}$ is an edge discriminator for $\mathcal H_r(\pmb a)$, which attains the lower bound in Equation (\ref{eqn:partite}). Therefore, $\lambda_{\mathrm{OPT}}=\lambda_{\mathcal H_r(\pmb a)}$, and the proof is completed.\\

\begin{remark}For the special case where $m_1=m_2=\ldots=m_r=m$, the weight of the optimal edge-discriminator for a complete $r$-partite hypergraph is $\frac{m(m^r + 1)}{2}$, which matches the upper bound in Proposition \ref{th:r+1} up to a constant factor, like the complete $r$-uniform hypergraph.
\end{remark}

\section{Non-Attainable Optimal Weights: Proof of Theorem \ref{th:no_function}}
\label{sec:nonattainable}

In the previous sections we have shown that optimal edge-discriminator on a hypergraph $\mathcal H=(\mathcal V, \mathcal E)$ with $n$ hyperedges can have weight at most $n(n+1)/2$ and this is attained when the $n$ hyperedges in $\mathcal H$ are mutually disjoint. Moreover, there exists a hypergraph for which the weight of the optimal edge-discriminator is $n$, as demonstrated in Theorem \ref{th:lower_bound}. This raises the question that whether all weights between $n$ and $n(n+1)/2$ are attainable. More formally, we are interested to know that given any integer $w\in [n, n(n+1)/2]$, whether there exists a hypergraph $\mathcal H(n, w)$ on $n$ edges such that $w$ is the weight of the optimal edge-discriminator on $\mathcal H(n, w)$. We say that $w$ is {\it attainable} if there exists a hypergraph $\mathcal H$ on $n$ edges such that the weight of the optimal edge-discriminator on $\mathcal H$ is $w$. An integer $w \in [n, n(n+1)/2]$ is said to be {\it non-attainable} if it is not attainable. In this section we prove Theorem \ref{th:no_function} by showing that for all $n\geq 3$, the weight $n(n+1)/2-1$ is non-attainable.

\subsection{Proof of Theorem \ref{th:no_function}}

For $n=3$, the result can be proved easily by considering the different possible distinct hypergraphs on 3 edges.

Now, fix an integer $n \geq 4$. We prove the theorem by contradiction. Assume that there exists a hypergraph $\mathcal H_O=(\bbV, \bbE)$, with $|\mathcal V|=m$ and $|\mathcal E|=n$, such that $\omega_0(\mathcal H_O)=n(n+1)/2-1$. Fix any ordering $\nu$ on $\bbV$. From the proof of Theorem \ref{th:hypergraph_main} we get an edge discriminator $\lambda'$ for $\mathcal H_O$, such that $\omega_{\lambda'}(\mathcal V)\geq \frac{n(n+1)}{2}-1$. 


We now have the following observation:

\begin{obs}
For any fixed ordering $\nu$ on $\bbV$, either of the following statements must be true:
\begin{description}
\item[{\it (i)}] There exists $j\in[m]$ such that $\pi(\nu_{j})=2$ and $\pi(\nu_k)=1$ for all $k\in [m]\backslash\{j\}$.
\item[{\it (ii)}] $\pi(\nu_k)=1$ for all $k\in [m]$.
\end{description}
\label{ob:cases}
\end{obs}

\begin{proof} Suppose we have an ordering $\nu$ such that $\pi(\nu_{j})\ge3$ for some $j$. Without loss of generality, we assume that $j= |\mathcal V|=m$ because otherwise we can work with a new ordering $\nu'$ such that $ \nu_{j}=\nu'_{m}$.

Let $\nu_{m}\in  E_x\cap  E_y\cap  E_z$. We start with the function $\kappa(v)=\delta_{\nu_{m}}$ which takes the value $1$ at $\nu_{m}$ and $0$ everywhere else. From Lemma \ref{lm:algogeneral} we  
get an edge-discriminator $\lambda_{\kappa, \nu}$ such that $\sum_{k=1}^{m} \lambda_{\kappa, \nu}(\nu_{k})\leq n(n+1)/2-2$, which contradicts the hypothesis that the optimal edge discriminator has weight $n(n+1)/2-1$.

Now suppose there are two numbers $i$ and $j$ such that both $\pi(\nu_{i})$ and $\pi(\nu_{j})$ are  equal to $2$. Then applying the same argument as above and starting with the function  $\lambda=\delta_{\nu_{i}}+\delta_{\nu_{j}}$ and using Lemma \ref{lm:algogeneral} we get an edge-discriminator $\lambda_{\kappa, \nu}$ such that $\sum_{k=1}^{m} \lambda_{\kappa, \nu}(\nu_{k})\leq n(n+1)/2-2$, which gives us a contradiction. Hence the proof is complete. 
\end{proof}

Using the above observation we now formulate the following important lemma:

\begin{lem}
No vertex in $\mathcal V$ can be incident on more than to $2$ hyperedges in $\mathcal E$.
\label{lm:less3}
\end{lem}
\begin{proof}If possible, suppose that there exists a vertex $v_o\in \mathcal V$ such that $v_o$ is incident on
$\ell~(\geq 3)$ hyperedges in $\mathcal E$. Define the ordering $\nu':[m]\rightarrow \bbV$, where $|\bbV|=m$, such that $\nu'_m:=\nu'(m)=v_o$. Therefore, $v_o$ must be the maximal vertex of all the $\ell$ hyperedges incident on it, that is, $\pi(\nu'_m)=\ell\geq 3$. This contradicts Observation \ref{ob:cases} and the proof of the lemma follows.

\end{proof}


Now, suppose that the second possibility in Observation \ref{ob:cases} holds for some ordering $\nu$, that is, $\pi(\nu_k)=1$ for all $k\in [m]$. Clearly, $\mathcal H_O$ cannot be the hypergraph in which all the $n$ are hyperedges disjoint. Therefore, we may assume that at least a pair of hyperedges intersect. Now, similar to the proof of Lemma \ref{lm:less3}, we can define a new order $\nu'$ on the vertices such that $\nu'_m:=\nu'(m)\geq 2$, and the problem reduces to the first possibility of Observation \ref{ob:cases} with respect to the ordering $\nu'$.

Therefore, it suffices to consider the first possibility in Observation \ref{ob:cases}, that is, there exists $j\in[m]$ such that $\pi(\nu_{j})=2$ and $\pi(\nu_k)=1$ for all $k\in [m]$ and  $k\ne j$. Let $F$ and $G$ be the two hyperedges having $\nu(i_j)$ as the maximal vertex. We now have the following lemma:

\begin{lem}
For any two hyperedges $A, B\in\bbE\backslash\{F, G\}$ $A\cap B=\emptyset$.
\label{lm:n(n+1)/2-1_disjoint}
\end{lem}

\begin{proof}
Suppose there exists $v_o\in \mathcal V$ such that $v_o\in A\cap B$. Now, from Lemma \ref{lm:less3}, $v_o \notin F \cup G$. Similarly, as $\nu_{j}\in F \cap G$ we have $\nu_{j}\notin A \cup B$.

Let us consider a new ordering
$\nu'$ on $\bbV$ such that $\nu'_{m-1}:=\nu'(m-1)=\nu_{j}$ and $\nu'_m:=\nu'(m)=v_o$.

Therefore, $\nu'(A)=\nu'(B)=v_o$, and so $\pi(\nu'_m)\geq 2$. Moreover, as $v_o\notin F\cup G$ and $\nu_{j}\in F\cap G$, $\nu'(F)=\nu'(G)=\nu_{j}$. This means that $\pi(\nu'_{m-1})\geq 2$.

This contradicts Observation \ref{ob:cases} and the result follows. 
\end{proof}

The above lemma helps us to deduce a necessary configuration of the hypergraph $\mathcal H_O$. This can be visualized by Figure \ref{fig:structure_diagram} and is summarized in the following lemma, the proof of which is immediate from Lemma \ref{lm:less3} and Lemma \ref{lm:n(n+1)/2-1_disjoint}.

\begin{lem}The set of hyperedges in $\mathcal E\backslash \{F, G\}$ can be partitioned into two disjoint sets $\mathcal F_A=\{A_1, A_2, \ldots, A_s\}$ and $\mathcal F_B=\{B_1, B_2, \ldots, B_t\}$, with $s+t=n-2$, such that:
\begin{description}
\item[{\it (i)}]$A_i\cap A_j=\emptyset$ and $A_i\cap (F \cup G)=\emptyset$ for distinct indices $i, j\in [s]$,
\item[{\it (ii)}]$B_i\cap B_j=\emptyset$, $B_i\cap (F \cup G)\ne \emptyset$, and $B_i\cap (F \cap G)=\emptyset$ for distinct indices $i, j\in [t]$. 
\end{description}
\label{lm:n(n+1)/2-1_structure}
\end{lem}

Using properties of this special configuration, we now construct an edge-discriminator $\lambda_1$  on $\mathcal H_O$, such that, $\omega_{\lambda_1}(\mathcal H_O)< n(n+1)/2-1$.

\begin{figure*}[h]
\centering
\begin{minipage}[c]{1.0\textwidth}
\centering
\includegraphics[width=5.25in]
    {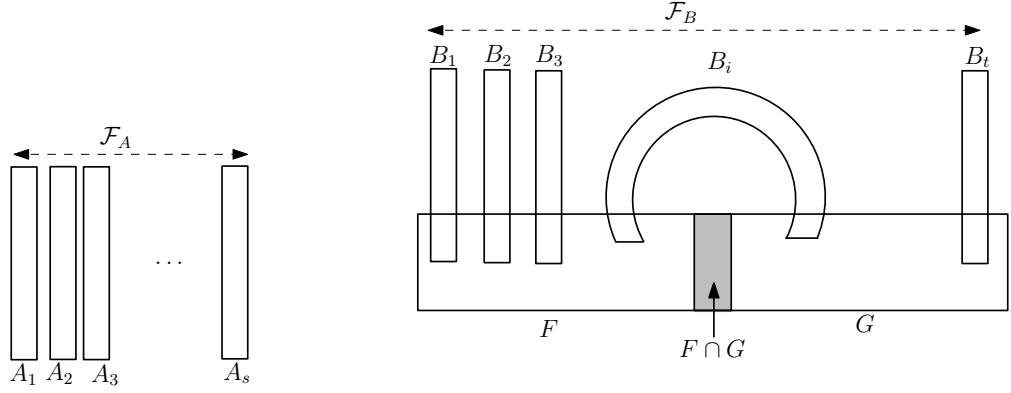}\\
\end{minipage}%
\caption{Illustration for the proof of Theorem \ref{th:no_function}.}
\label{fig:structure_diagram}
\end{figure*}

\begin{description}
\item[{\it Case} 1:]$s=n-2$. Then $|\mathcal E_B|=0$ and $F$ and $G$ are the only two
intersecting hyperedges in $\mathcal H_O$. We define $\lambda_1:\mathcal V\rightarrow \mathbb Z^{+}\cup \{0\}$ as follows:
$$\begin{array}{ll}
    \lambda_1(\nu(A_i))=i, & \hbox{for}~i\in [s]; \\
    \lambda_1(\nu(F\cap G))=n-1; &  \\
    \lambda_1(\nu(F\Delta G))=1; & \\
    \lambda_1(x)=0, & \hbox{otherwise.}
  \end{array}$$
It is clear that $\lambda_1$ is an edge-discriminator on $\mathcal H_O$, and
\begin{eqnarray}
\omega_{\lambda_1}(\mathcal H_O)=\sum_{v\in \mathcal V}{\lambda_1(v)}&=&\sum_{i=1}^{n-2}\omega_{\lambda_1}(A_i)+
\omega_{\lambda_1}(F)+\omega_{\lambda_1}(G)-\omega_{\lambda_1}(F\Delta G)\nonumber\\
&=& \frac{n(n-1)}{2}-1<\frac{n(n+1)}{2}-1. \nonumber
\label{eqn:n(n+1)/2-1_lambda_o_1}
\end{eqnarray}

\item[{\it Case} 2:]$s<n-2$. In this case there are $n-s-2$ hyperedges that intersect with $F \Delta G$. We define $\lambda'_1:\mathcal V\rightarrow \mathbb Z^{+}\cup \{0\}$ as follows:
$$\begin{array}{ll}
    \lambda'_1(\nu(A_i))=i, & \hbox{for}~i\in [s]; \\
    \lambda'_1(\nu(B_i\cap(F\Delta G)))=i+s; & \hbox{for}~ i\in [t-1];  \\
    \lambda'_1(x)=0, & \hbox{otherwise.}
  \end{array}$$
Now, we look at $$q=\left |\sum_{v\in F}{\lambda'_1(v)}-\sum_{v\in G}{\lambda'_1(v)}\right |.$$
If $q=n-1$ then we define $\lambda_1: \mathcal V\rightarrow \mathbb Z^{+}\cup \{0\}$ as
$$\begin{array}{ll}
    \lambda_1(\nu(A_i))=i, & \hbox{for}~i\in [s]; \\
    \lambda_1(\nu(B_i\cap(F\Delta G)))=i+s, & \hbox{for}~ i\in [t-1];  \\
    \lambda_1(\nu(B_t\cap(F \Delta G)))=n-2; & \\
    \lambda_1(\nu(F \cap G))=n-1; & \\
   \lambda_1(x)=0, & \hbox{otherwise.}
  \end{array}$$
If $q\ne n-1$ then we define
$$\begin{array}{ll}
    \lambda_1(\nu(A_i))=i, & \hbox{for}~i\in [s]; \\
    \lambda_1(\nu(B_i\cap(F\Delta G)))=i+s, & \hbox{for}~ i\in [t-1];  \\
    \lambda_1(\nu(B_t\cap(F \Delta G)))=n-1; & \\
    \lambda_1(\nu(F \cap G))=n; & \\
   \lambda_1(x)=0, & \hbox{otherwise.}
  \end{array}$$
It is again easy to see that $\lambda_1$ is an edge-discriminator on $\mathcal H_O$ and as $n>3$,
$$\sum_{v\in \mathcal V}{\lambda_1(v)}\leq\frac{n(n-1)}{2}+2<\frac{n(n+1)}{2}-1.$$
\end{description}

Therefore, $\lambda_1$ is an edge-discriminator on $\mathcal H_O$, such that, $\omega_{\lambda_1}(\mathcal H_O)< n(n+1)/2-1$. This contradicts our assumption that $\omega_0(\mathcal H_O)=n(n+1)/2-1$ and the proof of Theorem \ref{th:no_function} follows.

\section{Geometric Set Discrimination and Potential Applications}
\label{sec:application}

In this section we show how hypergraph edge-discriminators can be used to differentiate a collection of regions in $\mathbb R^d$. Consider a finite collection of regions $\mathcal R=\{R_1, R_2, \ldots, R_n\}$ in $\mathbb R^d$, where a {\it region} is a subset of $\mathbb R^d$. Given any $n$-tuple $(\epsilon_1, \epsilon_2, \ldots, \epsilon_n)\in \{0, 1\}^n$, define $\mathcal R(\epsilon_1, \epsilon_2, \ldots, \epsilon_n)=\bigcap_{i=1}^n R_i^{\epsilon_i}$, where $R_i^0=R_i$ and $R_i^1=\mathbb R^d\backslash R_i$, for $i\in [n]$. Also for $i \in [n]$, define $E_i=\bigcup_{(\epsilon_1, \epsilon_2, \ldots, \epsilon_n)\in A_i}\mathcal R(\epsilon_1, \epsilon_2, \ldots, \epsilon_n)$, where $A_i=\{(\epsilon_1, \epsilon_2, \ldots, \epsilon_n)\in \{0, 1\}^n: \epsilon_i=1\}$. The {\it geometric hypergraph} generated by $\mathcal R$, to be denoted by $\mathcal H(\mathcal R)$, is the hypergraph $(\bbV_{\mathcal R}, \bbE_\mathcal R)$, where $\bbV_{\mathcal R}=\{\mathcal R(\epsilon_1, \epsilon_2, \ldots, \epsilon_n): (\epsilon_1, \epsilon_2, \ldots, \epsilon_n)\in \{0, 1\}^n\}$ and $\mathcal E_\mathcal R=\{E_1, E_2, \ldots, E_n\}$.


An edge-discriminator for the geometric hypergraph $\mathcal H(\mathcal R)$ is a finite set $M\subset \mathbb R^d$ such that $|R_i\cap M|> 0$, for $i\in [n]$, and $|R_i\cap M|\ne |R_j\cap M|$, for all $i\ne j \in [n]$. The set $M$ is called the {\it geometric discriminator} for $\mathcal R$. The optimal edge-discriminator on $\mathcal H(\mathcal R)$ is the geometric discriminator of the least cardinality, and will be called the {\it optimal geometric discriminator} of $\mathcal R$. The problem of finding the optimal geometric discriminator for a geometric hypergraph, generated by a finite collection of regions in $\mathbb R^d$, will be called the {\it Geometric Set Discrimination Problem}.

Geometric set discrimination poses many interesting computational geometry problems. These include devising efficient algorithms or proving hardness results, particularly when the regions consist of intervals in $\mathbb R^1$, or rectangles or circles in $\mathbb R^2$. These algorithmic questions are left for future research. However, we shall discuss three simple examples of geometric set discrimination which will provide instructive insights into the properties of edge-discriminators and corroborate some of our earlier results.

\begin{figure*}[h]
\centering
\begin{minipage}[c]{0.33\textwidth}
\centering
\includegraphics[width=0.75in]
    {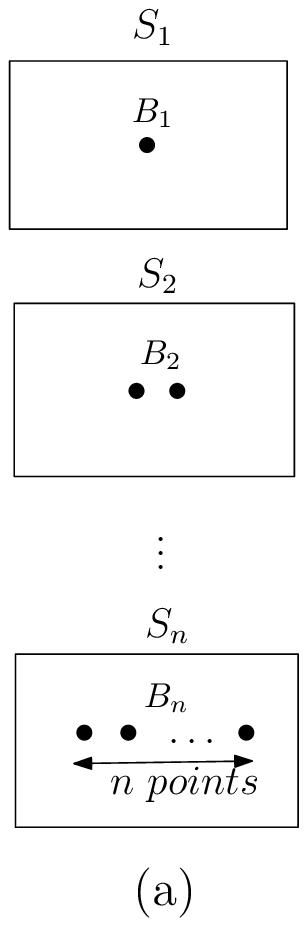}\\
\end{minipage}%
\begin{minipage}[c]{0.33\textwidth}
\centering
\includegraphics[width=1.95in]
    {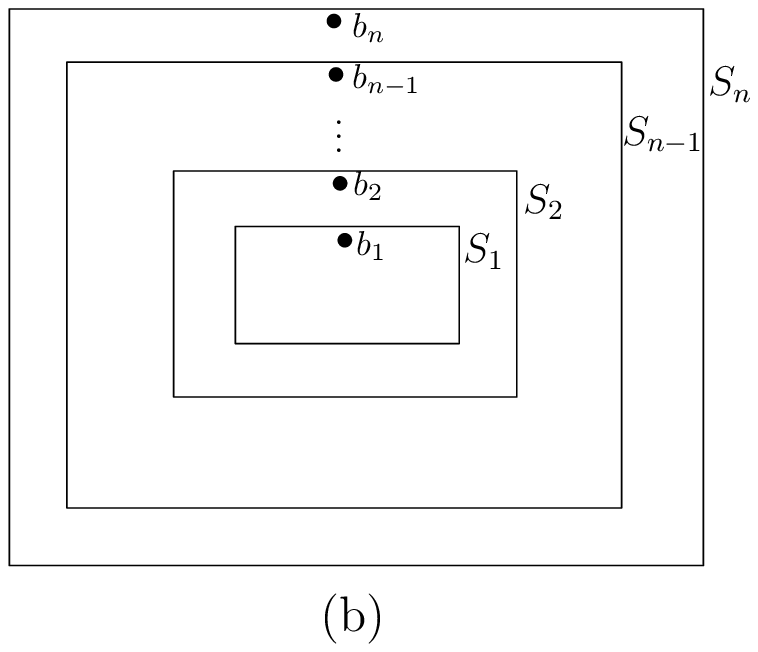}\\
\end{minipage}
\begin{minipage}[c]{0.33\textwidth}
\centering
\includegraphics[width=1.45in]
    {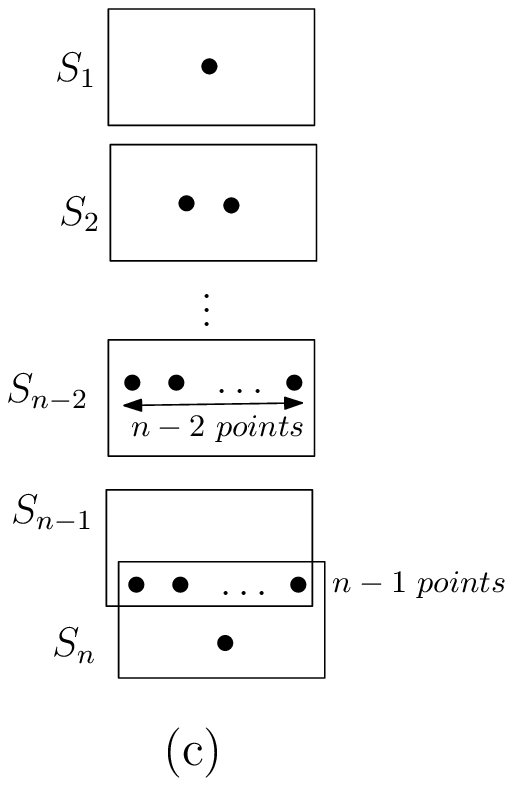}\\
\end{minipage}
\caption{Examples of geometric set discrimination}
\label{fig1}
\end{figure*}

\begin{description}
\item[{\it Example} 1:] We have shown that the upper bound on the weight of an edge-discriminator proved in Theorem \ref{th:hypergraph_main} is attained if and only if the hypergraph has $n$ disjoint edges. The geometric hypergraph  generated by the set $\mathcal S=\{S_1, S_2, \ldots, S_n\}$ of mutually disjoint axis-aligned squares (Figure \ref{fig1}(a)) 
is such an example. Let $B_i\subset \mathbb R^2$ be any set of distinct $i$ points in the interior of $S_i$. It is trivial to see that the set $B=\{B_1, B_2, \ldots, B_n\}$ is the optimum geometric discriminator of $\mathcal S$ and $\omega_0(\mathcal H_\mathcal S)=\frac{n(n+1)}{2}$. 

\item[{\it Example} 2:]Consider the set $\mathcal S=\{S_1, S_2, \ldots, S_n\}$ of axis-aligned squares such that $S_i\subset S_{i+1}$ for all $i\geq 1$. Let $B=\{b_1, b_2, \ldots, b_n\}$, where $b_i \in S_i\backslash \bigcup_{j=1}^{i-1}S_j$ (Figure \ref{fig1}(b)). Clearly, $B$ is the optimum edge-discriminator of the geometric hypergraph $\mathcal H_{\mathcal S}$. Since $|B|=n$, the optimum-weight of the edge-discriminator of $\mathcal H_\mathcal S$ attains the lower bound in Theorem \ref{th:lower_bound}. 

\item[{\it Example} 3:]Consider the set $\mathcal S=\{S_1, S_2, \ldots, S_n\}$ of axis-aligned squares, such that $S_i \cap S_j=\emptyset$ for all $i \ne j \in [n-1]$, and $S_i\cap S_n=\emptyset$ for all $i\in [n-2]$, and $S_{n-1}\cap S_n\ne \emptyset$ (see Figure \ref{fig1}(c)). Let $B_i$ be any set of $i$ points in the interior of $S_i$, for $i\in [n-2]$. Let $B_{n-1}$ be any set of $n-1$ points in the interior of $S_{n-1}\cap S_n$ and $b_n$ is any point in the interior of $S_n\backslash S_{n-1}$. It is easy to see that set $B=\{B_1, B_2, \ldots, B_{n-1}, b_n\}$ is the optimum geometric discriminator of $\mathcal S$, with $\omega_0(\mathcal H_\mathcal S)=n(n-1)/2+1$. Note that in this example the $n$ hyperedges are almost disjoint, but the weight of the optimal edge-discriminator is $\frac{n(n-1)}{2}+1=\frac{n(n+1)}{2}-(n-1)$. In fact, this example leads us to conjecture that all integer values in $\mathcal N:=\left[\frac{n(n-1)}{2}+2, \frac{n(n+1)}{2}-1\right]$ are non-attainable.  In Theorem \ref{th:no_function} we only show that $\frac{n(n+1)}{2}-1$ is non-attainable. It might be possible to generalize the proof of Theorem \ref{th:no_function} to prove that weights like $n(n+1)/2-a$ are non-attainable, for small constant values of $a$. However, proving it for all integer values in $\mathcal N$, that is, for all $a\in [n-2]$ appears to be challenge.
\end{description}

Geometric set discrimination problems have potential applications in unique image indexing in large database \cite{euler_smc,image_indexing}, where the emphasis is specially given on deciding whether a particular image
exists in the database, rather than on finding the similarity matches of the given image. A novel method for image indexing using only the number of connected components, the number of holes, and the Euler number of an image was proposed by Biswas et al. \cite{image_indexing}. A {\it connected component} of a digital binary image is a subset of maximal size such that any two of its pixels can be joined by a connected curve, in 8-connectivity, lying entirely in the subset. A {\it hole} in a digital image is a region of the background, which is a connected component in 4-connectivity and is completely enclosed by the object. The {\it Euler number} of an image is defined as the number of connected components minus the number of holes in the image. If $C$ and $H$ denote the number of connected components and the number of holes in a digital image, respectively, then its Euler number $E=C-H$ \cite{euler_smc,gonzalez_woods,pratt}. The ordered pair $(C, H)$ is called the {\it Euler pair} of a digital image. It is apparent that two or more images may have the same value of the Euler pair, and hence this feature alone often cannot uniquely characterize an image in a large database. One way to disambiguate the features is to deploy a mask image \cite{image_indexing} as follows: We assume that each image is given as a $(k_1 \times k_2)$ binary pixel matrix. Let us consider $n$ images $I_1, I_2, \ldots, I_n$, each having the same Euler pair. In order to discriminate them, another binary image $M$ called mask is to be constructed such that the Euler pair of the $n$ images $I_1\odot M, I_2\odot M, \ldots, I_n \odot M$ are mutually distinct, where $\odot$ denotes bitwise Boolean operation, like XOR or AND between the corresponding bits of the two pixel matrices.

\begin{figure*}[h]
\centering
\begin{minipage}[c]{1.0\textwidth}
\centering
\includegraphics[width=4.5in]
    {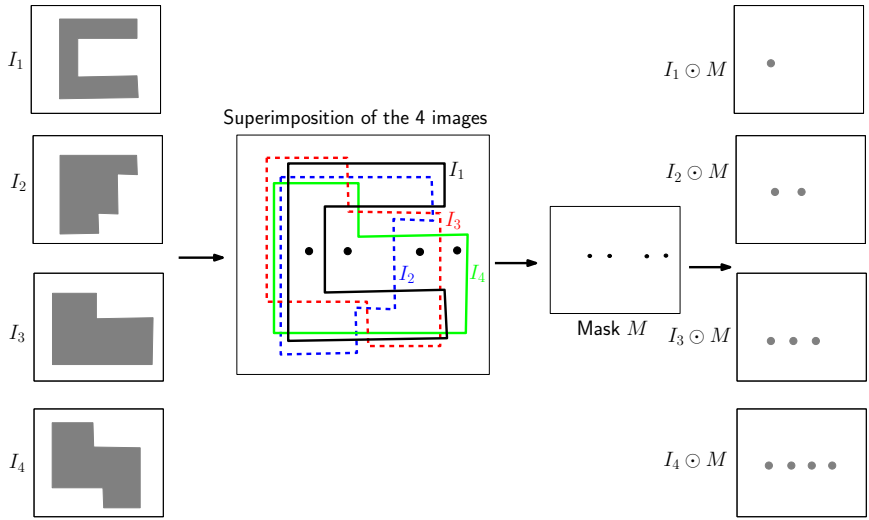}\\
\end{minipage}%
\caption{Unique image-indexing by geometric-set discrimination.}
\label{fig:image_indexing}
\end{figure*}

As it turns out, finding a simple mask of a given set of images is a challenging problem. Biswas et al. \cite{image_indexing} provided an iterative heuristic based on few synthetic pseudo-random masks. We now show that finding a mask for a set of images can be modeled as an instance of the hypergraph edge-discrimination for a collection of geometric regions in $\mathbb R^2$. Consider the images $I_1, I_2, \ldots, I_n$, superimposed on each other in the same frame, as subsets of $\mathbb R^2$. Suppose $M$ is a geometric discriminator for this collection of regions. Then the binary image corresponding to $M$ is a mask for the set of images under the bitwise Boolean AND operation. The process is illustrated with 4 binary images in the Figure \ref{fig:image_indexing}.\footnote{As binary images are actually subsets of the discrete space $\mathbb Z^2$, the mask $M$ should be a subset of $\mathbb Z^2$. As a result, several technical difficulties may arise while trying to obtain a geometric discriminator for a set of binary images containing holes. These problems need to be handled separately and they are not of interest to this paper. Here we use the unique indexing problem just as a motivation for the edge-discrimination problem on hypergraphs.} The mask corresponding to the optimal geometric discriminator is the simplest in the sense that the pixel matrix has the least number of ones.

\section{Conclusions}
\label{sec:conclusions}

In this paper we introduce the notion of hypergraph edge-discrimination and study its properties. We have shown that given any hypergraph $\mathcal H=(\mathcal V, \mathcal E)$, with $|\mathcal V|=m$ and $|\mathcal E|=n$, $\omega_0(\mathcal H)\leq n(n+1)/2$, and the equality holds if and only if the elements of $\mathcal E$ are mutually disjoint. For $r$-uniform hypergraphs, using properties of $B_h$-sets, we prove that $\omega_0(\mathcal H)\leq m^{r+1}+o(m^{r+1})$, and the bound is attained by a complete $r$-uniform hypergraph up to a constant factor.

Moreover, it is easy to see that for any hypergraph $\mathcal H=(\mathcal V, \mathcal E)$ and any edge-discriminator $\lambda$ on $\mathcal H$, $\omega_\lambda(\mathcal V)\geq \max\{n, \delta(\delta+1)/2\}$, where $|\mathcal E|=n$ and $\delta$ is the size of the maximum matching in $\mathcal H$. This motivated us to consider the question of attainability of weights: Given any integer $w\in [n, n(n+1)/2]$, we ask whether there exists a hypergraph $\mathcal H(n, w)$ with $n$ hyperedges such that the weight of the optimal edge-discriminator on $\mathcal H(n, w)$ is $w$. We answer this question in the negative by proving that there exists no hypergraph on $n~(\geq 3)$ hyperedges such that the weight of the optimal edge-discriminator is $n(n+1)/2-1$. The problem of attainability of weights appears to be a very interesting problem which might lead to surprising consequences. 


Computing optimal edge-discriminators for special hypergraphs are also interesting combinatorial problems. We have computed the optimal edge-discriminators for paths, cycles, and the complete $r$-partite hypergraph. Finding optimal edge-discriminators appear to be quite difficult even for very simple graphs, in particular if the graph is not regular. Interesting graphs that might be considered for future research are the wheel and the hypercube.

Finally, as mentioned in the previous section, one of the major problems for future research is the algorithmic study of the geometric set-discrimination problem.

\end{document}